\documentclass[11pt,a4paper,oneside]{amsart}

\usepackage{amssymb}
\usepackage[shortlabels]{enumitem}
\setlist{leftmargin=*}
\setlist[enumerate]{label=(\arabic*)}
\usepackage{geometry}
\usepackage{todonotes}
\usepackage[all]{xy}

\usepackage[hyperindex=true,bookmarks=true,bookmarksnumbered=true,debug=true]{hyperref}
\usepackage[capitalize]{cleveref} 
\usepackage{autonum}

\makeatletter
\def\@seccntformat#1{%
  \protect\textup{%
    \protect\@secnumfont
    \expandafter\protect\csname format#1\endcsname 
    \csname the#1\endcsname
    \protect\@secnumpunct
  }%
}

\newtheoremstyle{convenientthm}%
  {3pt}
  {3pt}
  {\itshape}
  {}
  {\bfseries}
  {.}
  {.5em}
  {\thmnumber{#2 }\thmname{#1}\thmnote{. #3\addcontentsline{toc}{subsection}{\tocsubsection {}{#2}{#1. #3}}}}

\newtheoremstyle{convenientplain}%
  {3pt}
  {3pt}
  {}
  {}
  {\bfseries}
  {.}
  {.5em}
  {\thmnumber{#2 }\thmname{#1}\thmnote{. #3\addcontentsline{toc}{subsection}{\tocsubsection {}{#2}{#1. #3}}}}

\theoremstyle{convenientthm}
\newtheorem{proposition}[subsection]{Proposition}
\newtheorem*{proposition*}{Proposition}
\newtheorem{theorem}[subsection]{Theorem}
\newtheorem*{theorem*}{Theorem}
\newtheorem{lemma}[subsection]{Lemma}
\newtheorem*{lemma*}{Lemma}
\newtheorem{corollary}[subsection]{Corollary}
\newtheorem*{corollary*}{Corollary}
\newtheorem*{conjecture*}{Conjecture}
\newtheorem{definition}[subsection]{Definition}
\newtheorem{definition*}{Definition}

\theoremstyle{convenientplain}
\newtheorem{example}[subsection]{Example}
\newtheorem{remark}[subsection]{Remark}

\newcommand{\loc}{{\rm loc}}

\newcommand{\comment}[1]{}

\newcommand{\Diff}{\operatorname{Diff}}

\newcommand{\on}{\operatorname}
\newcommand{\Id}{\operatorname{Id}}

\def\o{\,\circ\,}
\def\X{\mathfrak X}
\def\al{\alpha}
\def\be{\beta}

\def\ep{\varepsilon}

\def\et{\eta}

\def\rh{\rho}
\def\si{\sigma}
\def\ta{\tau}
\def\ph{\varphi}

\def\ps{\psi}
\def\om{\omega}
\def\Ga{\Gamma}

\def\Om{\Omega}
\def\i{^{-1}}
\def\x{\times}
\def\p{\partial}
\let\on=\operatorname
\def\L{\mathcal L}
\def\Diff{\operatorname{Diff}}

\makeatletter
\@namedef{subjclassname@2020}{\textup{2020} Mathematics Subject Classification}
\makeatother

\title{Regularity and completeness of half-Lie groups}
\author[M.~Bauer]{Martin Bauer}
\address{Florida State University, Tallahassee and University of Vienna, Vienna}
\email{bauer@math.fsu.edu}
\author[P.~Harms]{Philipp Harms}
\address{Research Industrial Systems Engineering (RISE) Schweiz AG, 3053 Münchenbuchsee,  Switzerland}
\email{philipp.harms@rise-world.com}
\author[P.~W.~Michor]{Peter W.~Michor}
\address{University of Vienna, Vienna}
\email{peter.michor@univie.ac.at}

\begin{document}
\maketitle
\begin{abstract}
Half Lie groups exist only in infinite dimensions: They are smooth manifolds and topological groups such that right translations are smooth, but left translations are merely required to be continuous.
The main examples are groups of $H^s$ or $C^k$ diffeomorphisms and semidirect products of a Lie group with kernel an infinite dimensional representation space. Here, we investigate mainly Banach half-Lie groups, the groups of their $C^k$-elements, extensions, and right invariant strong Riemannian metrics on them: surprisingly the full Hopf--Rinow theorem holds, which is wrong in general even for Hilbert manifolds. 
\end{abstract}

\setcounter{tocdepth}{1}
\tableofcontents

\section{Introduction}

Infinite-dimensional Riemannian geometry can be traced back all the way to the birthplace of Riemannian geometry, Riemann's Habilitationsschrift \cite{riemann2016hypotheses}, in which he already mentioned the potential need of considering infinite-dimensional manifolds. Later on, these proved to be central in several fields, including mathematical hydrodynamics in Arnold's geometric picture \cite{arnold1966, ebin1970groups} and functional data and shape analysis~\cite{younes2019shapes, srivastava2016functional, bauer2014overview}. Motivated by these applications, a significant amount of work has been dedicated to studying theoretical properties of infinite-dimensional geometric spaces. This revealed several astonishing phenomena, where well-known results of finite-dimensional Riemannian geometry cease to hold in infinite dimensions. Some notable examples are the non-existence of Christoffel symbols~\cite{bauer2014homogenous}, vanishing geodesic distance~\cite{eliashberg1993bi, michor2005vanishing, jerrard2019vanishing, bauer2020vanishing}, and the failure of the theorem of Hopf--Rinow~\cite{atkin1975hopf, ekeland1978hopf, atkin1997geodesic}. 

We consider a special class of infinite-dimensional manifolds, namely, half-Lie groups. These are topological groups and smooth manifolds such that all right translations are smooth. Alternatively, by considering the group of inverses, one may require all left translations to be smooth. (Throughout this article, smooth means $C^\infty$.) Banach half-Lie groups are special cases of one-sided differentiable groups, which have been studied in the context of Hilbert's 5th problem  by Birkhoff~\cite{birkhoff1938analytical} and Enflo~\cite{enflo1969topological}; see also~\cite{benyamini1998geometric}. The most important examples of half-Lie groups, also in the context of the above-mentioned applications in fluid dynamics and functional data or shape analysis, are groups of $C^k$ or $H^s$ diffeomorphisms, as first studied by Eells, Eliasson and Palais~\cite{Eells66, Eliasson67, Palais68}. More exotic diffeomorphism groups were studied in \cite{KMR14}, where also the name half-Lie group was coined. Another vast class of examples of half-Lie groups are semidirect products of Lie groups with representation spaces. Recently, Marquis and Neeb~\cite{MarquisNeeb18} studied Lie-theoretic properties of such semidirect products, namely, regularity properties and subspaces of differentiable vectors. We continue their investigations by studying general half-Lie groups, which are not necessarily semidirect products, from both Lie-theoretic and Riemannian perspectives.

The group structure and differentiable structure of a half-Lie group are by definition only partially compatible. On the positive side, the partial compatibility explains why half-Lie groups enjoy better regularity and completeness properties than general infinite-dimensional manifolds, and we show several results in this direction. On the negative side, the partial incompatibility explains some of the problems in infinite-dimensional group and representation theory. For example, the canonical trivialization of the tangent bundle of a half-Lie group may be non-smooth and even discontinuous. Consequently, the Eulerian and Lagrangian coordinates are incompatible, and the Lie bracket is only partially defined. Somewhat surprisingly, many half-Lie groups carry an additional geometric structure, which is compatible with both the group and differentiable structure, examples being right-invariant smooth connections and right-invariant smooth Riemannian metrics. This has far-reaching implications, which we exploit in our study of half-Lie groups.
Our main contributions are as follows:
\begin{itemize}
\item {\bf Differentiable elements in half-Lie groups:} In Section~\ref{sec:differentiable}, we study the subgroup $G^k$ of all $C^k$ elements in $G$, i.e., those elements in $G$ such that left translation is $C^k$. We then show for any $k\in \mathbb N$ that the set $G^k$ is a regular half-Lie group and that the inverse limit $k\to \infty$ is a regular Fr\'echet Lie group. This generalizes previous results of Marquis and Neeb~\cite{MarquisNeeb18}, who studied related questions for the case that $G$ is a semidirect product. In \cref{thm:semidirect} we discuss how our result strengthens theirs even in this special case. 

\item {\bf Regularity of half-Lie groups:} Next, in Section~\ref{sec:regularity}, we discuss smooth regularity of half-Lie groups, i.e., the  question if one can integrate smooth curves in the tangent space at the identity to smooth curves in the group by inverting the right logarithmic derivative. We are unable to show that each Banach half-Lie group $G$ itself is regular, although all known examples are regular. We can, however, show that any group $G^k$ for $k\geq 1$ is regular.

\item {\bf Extension theory for  half-Lie groups:} 
In Section~\ref{sec:extensions}, we show that the extension theory of Lie groups largely carries over to half-Lie groups. We can describe all extensions satisfying some conditions by extension data; see \cref{thm:non-split}. We describe split extensions (semidirect products; see \cref{lem:extensions:semidirect}) and central extensions (\cref{central_extensions}), and we are able to describe the subgroups of $C^k$-elements in an extension as another extension; see \cref{thm:Ek}.  As an application of the developed theory we present in \cref{sec:examples_extensions}  two examples of extensions in the context of diffeomorphism groups of fiber bundles.

\item {\bf Completeness on half-Lie groups:} In Section~\ref{sec:riemannian}, we study Riemannian metrics on half-Lie groups. The main result of this part shows under an additional technical property that strong right-invariant Riemannian metrics on half-Lie groups satisfy all completeness statements of the theorem of Hopf-Rinow, i.e., they are geodesically and metrically complete, and there exists a minimizing geodesic between any two points. Previously, only the geodesic completeness was known \cite{gay2015geometry}.
\end{itemize}

\subsection*{Notations and conventions} For a group $G$, we denote multiplication by $\mu:G\x G\to G$ with $\mu(x,y)=x.y = \mu_x(y) = \mu^y(x)$ indicating left and right translations. The results of this paper are valid for the following two choices of categories: (1)  Lie groups are group objects in the category of convenient smooth manifolds \cite{KrieglMichor97} and topological groups are with respect to the corresponding convenient topologies, or (2) Topological groups are group objects in the category of Hausdorff topological spaces, and Lie groups are groups objects in Bastiani $C^\infty$-manifolds modelled on locally convex spaces \cite{glockner2011infinite}. Note that in general, there are more morphisms and products carry a finer topology, under choice (1) compared to choice (2). For metrizable spaces or Fr\'echet manifolds there is no difference.  

\subsection*{Acknowledgements.}
We gratefully acknowledge support in the form of a Research in Teams stipend of the Erwin Schr\"odinger Institute Vienna. MB was partially supported by NSF grants DMS-1912037 and DMS-1953244 and by FWF grant FWF-P 35813-N. PH was supported by NRF Singapore grant NRF-NRFF13-2021-0012 and by NTU Singapore grant NAP-SUG.

\section{Half-Lie groups}
\label{sec:half_lie}
We start by introducing the central objects of the present article, half-Lie groups:
\begin{definition}[Half-Lie groups]
\label{def:half_lie}
A right (left) \emph{half-Lie group} is a smooth manifold, possibly infinite dimensional, whose underlying topological space is a topological group, such that right (left) translations are smooth.
We shall speak of Hilbert, Banach, Fr\'echet, etc. half-Lie groups to designate the nature of the modeling vector space.
A \emph{homomorphism} of half-Lie groups is a smooth group homomorphism.
\end{definition}

Lie groups are both right half-Lie groups and  left half-Lie groups with smooth multiplication and inversion. 
Every finite-dimensional half-Lie group is a Lie group by a result of Segal~\cite{segal1946topological}.
Every Banach half-Lie group with uniformly continuous multiplication is already a Banach Lie group.
This can be seen as a solution of Hilbert's 5th problem in infinite dimensions, due to Birkhoff~\cite{birkhoff1938analytical} and Enflo~\cite{enflo1969topological}, see also~\cite{benyamini1998geometric}.
Marquis and Neeb \cite{MarquisNeeb18} have collected a long list of examples of half-Lie groups.
We next present two important special cases.

\begin{example}[Diffeomorphism groups]
\label{ex:diffeo}
The main motivating examples for the present investigation of half-Lie groups are diffeomorphism groups with finite regularity.
These appear naturally in shape analysis~\cite{younes2019shapes,srivastava2016functional,bauer2014overview} and mathematical fluid dynamics~\cite{arnold1998topological,ebin1970groups,kolev2017local}.
If $(M,g)$ is a finite-dimensional compact Riemannian manifold or an open Riemannian manifold of bounded geometry, then the diffeomorphism group $\Diff_{H^s}(M)$ of Sobolev regularity $s>\dim(M)/2 +1$ is a half-Lie group.
Likewise, the groups $\Diff_{W^{s,p}}(M)$ for $s>\dim(M)/p +1$ and $\Diff_{C^k}(M)$ for $1\le k <\infty$ are half-Lie groups.
However, they are not Lie groups because left multiplication is non-smooth.
\end{example}

\begin{example}[Group representations]
\label{ex:representation}
Let $\rho: G\to U(H)$ be a representation of a Banach Lie group $G$ on an infinite-dimensional Hilbert space $H$, which is continuous as a mapping $G\times H\ni (g,h)\mapsto \rho(g)h\in H$.
Then the right semidirect product $G\ltimes H$ with operations
\begin{equation}
(g_1,h_1).(g_2,h_2) = (g_1g_2, \rho(g_2^{-1})h_1 +h_2),
\qquad
(g,h)\i=(g^{-1}, -\rho(g)h)
\end{equation}
is a right half-Lie group but not a Lie group.
This class of examples has been studied in detail by Marquis and Neeb~\cite{MarquisNeeb18}.
In their work, the roles of left and right translations interchanged compared to ours, but this makes no difference as one may always pass to the group of inverses.
\end{example}

\begin{remark}[Continuity of left translations]
By definition, left translations on half-Lie groups are continuous.
This is not automatic.
Indeed, there are diffeomorphism groups where right translations are smooth, but left translations are discontinuous.
An example is the group 
\begin{align}
  &\Diff_{\mathcal B^{(M)}_\text{loc}}(\mathbb R^n) =\big\{\Id+f: f\in \mathcal B^{(M)}_{\text{loc}}(\mathbb R^n,\mathbb R^n),
    \inf_{x\in \mathbb R^n} \det(\mathbb I_n + df(x)) >0\big\}
\end{align}
modeled on a space of ultradifferentiable functions
\begin{align}
\mathcal B^{(M)}_{\text{loc}}(\mathbb R^n) := \mathcal E^{(M)}(\mathbb R^n)\cap \mathcal B(\mathbb R^n),
\end{align}
where $(M)$ is a strongly non-quasi\-analytic weight-sequence such that $M_{k+1}/M_k\nearrow \infty$ and $\mathcal E^{(M)}\supseteq \mathcal E^{\{(k!)^{1/2}\}}$, and where $\mathcal B(\mathbb R^n)$ is the Fr\'echet space of smooth functions with bounded derivatives. An example is $M_k:= (k!)^{1/2}$. See  \cite[Theorem 13.3]{KMR14} for more details.
\end{remark}

\section{Differentiable elements in half-Lie groups}
\label{sec:differentiable}

We next consider subsets of half-Lie groups where left multiplication has better differentiability properties than mere continuity.
Specifically, we consider $k$-fold Fr\'echet differentiability, denoted by $C^k$, of left multiplication.\begin{definition}[Differentiable elements]
\label{def:differentiable}
Let $G$ be a Banach right half-Lie group.
Then, $x \in G$ is called a $C^k$ element if the left translations $\mu_x,\mu_x\i:G\to G$ are $C^k$.
The set of all $C^k$ elements of $G$ is denoted by $G^k$.
\end{definition}

We do not know if the $C^k$ property of $\mu_x\i$ follows automatically from the $C^k$ property of $\mu_x$.
By the inverse function theorem, this is the case if $\mu_x$ has an invertible derivative at some (and hence any) point.
But we do not know this, so
we require $\mu_x\i$ to be $C^k$ in \cref{def:differentiable}.

Next, we will show that the set of $C^k$ elements in a Banach half-Lie group $G$ is again a Banach half-Lie group, provided that $G$ carries a right-invariant local addition.

\begin{definition}[Right-invariant local additions]
\label{def:addition}
Let $G$ be a right half-Lie group.
A local addition on $G$ is a smooth map $\ta:TG\supseteq V  \to G$, which is defined on an open neighborhood $V $ of the 0-section in $TG$, such that $\ta(0_x)=x$ for all $x\in G$ and $(\pi_{G},\ta):V \to G\times G$  is a diffeomorphism onto its range.
The local addition $\ta$ is called right-invariant if $T\mu^y(V) = V $ and $\ta \o T\mu^y = \mu^y \o \ta$ holds for all $y \in G$.
\end{definition}

\begin{remark}[Existence of right-invariant local additions]
Right-invariant local additions exist on many important examples of half-Lie groups:
\begin{itemize}
\item Any local addition on a compact manifold $M$ induces a right-invariant local addition on the half-Lie group $\Diff_{C^k}(M)$, $k\geq 1$.
\item Any local addition on a finite-dimensional manifold $M$ with bounded geometry induces a right-invariant local addition on the half-Lie group $\Diff_{H^k}(M)$, $k>\on{dim}(M)/2+1$.
\item The exponential map on a Banach half-Lie group, provided it exists and is smooth, of any right-invariant weak Riemannian metric is a local addition.
\item The semidirect products considered in \cite{MarquisNeeb18} are half-Lie groups with local additions.
\end{itemize}
\end{remark}

Local additions are closely related to linear connections, sprays, and geodesic structures, as explained in \cref{sec:local_additions}. 
The following theorem is the first main result of this article. 

\begin{theorem}[Differentiable elements]
\label{thm:Gk}
For any Banach right half-Lie group $G$ carrying a right-invariant local addition, the following statements hold:
\begin{enumerate}[(a)]
\item For any $k \in \mathbb N$, $G^k$ is a Banach half-Lie group.
\item \label{thm:Gk:b} The tangent space $T_eG^k$ is the set of all $X \in T_eG$ such that the right-invariant vector field $R_X:G\ni x \mapsto T_e\mu^x(X)\in TG$ is $C^k$.
\item The inclusion $G^k \to G$ is smooth.
\item \label{thm:Gk:d}
For any $\ell \in \mathbb N$, $G^{k+\ell}$ is a subset of $(G^k)^\ell$.
\item The smooth right-invariant local addition on $G$ induces a smooth right-invariant local addition on the subgroup $G^k$.
\end{enumerate}
\end{theorem}

The \lcnamecref{thm:Gk} is proven at the end of \cref{sec:CkGGG}.
The idea is to identify any $x \in G^k$ with the left-multiplication $\mu_x$, which belongs to the space of right-invariant $C^k$ diffeomorphisms on $G$.
This space is a Banach manifold, and the Banach manifold structure is inherited by $G^k$.

\begin{example}[Differentiable elements in diffeomorphism groups]
$(\Diff_{C^k}(M))^\ell=\Diff_{C^{k+\ell}}(M)$, for any closed manifold $M$. 
To see that any $g \in (\Diff_{C^k}(M))^\ell$ is $C^{k+\ell}$, we fix $x \in M$, $y \in T_xM$, and a vector field $h \in \mathfrak X(M)$ which is constant and equal to $y$ locally near $x$ in some chart. 
Then, in this chart, 
\begin{align}
\partial_s^k|_0\partial_t^\ell|_0\
g\o(\Id+th)(x+sy)
=
g^{(k+\ell)}(x)(y,\dots,y),
\end{align}
with continuous dependence on $x$ and $y$. 
Therefore, $g\in \Diff_{C^{k+\ell}}(M)$. 
Conversely, any $g\in \Diff_{C^{k+\ell}}(M)$ is a $C^\ell$ element in $\Diff_{C^k}(M)$ by \cref{lem:Ck_calculus}.
Thus, $G:=\Diff_{C^1}(M)$ satisfies $(G^k)^\ell=G^{k+\ell}$, i.e., equality holds in \cref{thm:Gk}.\ref{thm:Gk:d}.
\end{example}

The charts for $G^k$ restrict to charts for $G^\ell$ for any $\ell\geq k$.
This implies that the intersection $G^\infty=\bigcap_k G^k$ is an inverse limit of Banach (ILB) manifold in a sense similar to Omori~\cite{omori1979infinite}.

\begin{definition}[ILB manifold]
\label{def:ilb}
An ILB manifold is a manifold $M$ modeled on a Fr\'echet space $E$ with the following properties:
\begin{enumerate}[(a)]
\item $E=\bigcap_{k\in\mathbb N}E^k$ is the inverse limit of a chain of Banach spaces $E^k$, such that $E^{k+1}$ is continuously embedded in $E^k$, for each $k\in\mathbb N$.
\item There is a collection, indexed by $\al\in A$, of open sets $U_\al$ covering $M$, open sets $V_\al$ in $E^0$, and homeomorphisms $u_\al:U_\al\to E\cap V_\al$.
\item For any $\al,\be\in A$ with $U_\al\cap U_\be\neq \emptyset$, there are open subsets $V_{\al\be}$ and $V_{\be\al}$ in $E^0$ such that $u_\al\i(U_\al\cap U_\be)=V_{\al\be}\cap E$, $u_\be\i(U_\al\cap U_\be)=V_{\be\al}\cap E$, and $u_\al\i\o u_\be:V_{\be\al}\cap E\to V_{\al\be}\cap E$ extends to a smooth map $V_{\be\al}\cap E^k\to V_{\al\be}\cap E^k$, for all $k \in \mathbb N$.
\end{enumerate}
\end{definition}

Compared to this \lcnamecref{def:ilb}, Omori~\cite{omori1979infinite} additionally requires $E^{k+1}$ to be densely included in $E^k$.
This is not true in the present generality.
Indeed, there are examples where $G^1=\{e\}$ is trivial \cite{neeb2010differentiable}.
Moreover, the further properties in Omori's definition of ILB Lie groups \cite[Definition~III.3.1]{omori1979infinite} may also fail.
Nevertheless, the following statement holds:
\begin{lemma}[Smooth elements in half-Lie groups]
\label{lem:smooth_elements}
For any Banach half-Lie group $G$ carrying a right-invariant local addition, the set $G^\infty=\cap_{k\in\mathbb N} G^k$ of smooth elements in $G$ is an ILB manifold and a Lie group.
\end{lemma}

The \lcnamecref{lem:smooth_elements} is proven at the end of \cref{sec:CkGGG} using the explicit construction of the manifold charts for $G^k$.

Finally, we investigate Lie algebras of half-Lie groups.
The term Lie algebra is somewhat misleading because the Lie bracket is not globally defined, as shown next.

\begin{lemma}[Lie bracket on half-Lie groups]
\label{lem:lie_bracket}
Let $G$ be a Banach half-Lie group carrying a right-invariant local addition.
For any $k \in \mathbb N$, the Lie bracket
\begin{equation}
[\cdot,\cdot]:T_eG^{k+1}\x T_eG^{k+1}\to T_eG^k
\end{equation}
is well defined in the following three equivalent ways:
\begin{enumerate}[(a)]
\item\label{lem:lie_bracket:1} Any vectors $X,Y \in T_e G^{k+1}$ extend uniquely to right-invariant $C^1$ vector fields $R_X,R_Y$ on $G^k$, and $[X,Y] := dR_Y(e)(X)-dR_X(e)(Y)$, where the right-hand side is interpreted in a chart around $e \in G^k$.
\item\label{lem:lie_bracket:2} The derivation $f\mapsto (R_XR_Yf-R_YR_Xf)(e)$ on smooth real-valued functions defined near $e \in G^k$ is the derivative in the direction of a vector in $T_eG^k$, which is denoted by $[X,Y]$.
\item\label{lem:lie_bracket:3} The vector field $R_X$ has a $C^1$ flow $\on{Fl}^{R_X}:\mathbb R\times G^k \to G^k$, and consequently, the derivative $[X,Y] := -\partial_t|_0 (\on{Fl}^{R_X}_t)^* R_Y(e) \in T_eG^k$ exists.
\end{enumerate}
\end{lemma}

\begin{proof}
\begin{enumerate}[(a),wide]
\item By \cref{thm:Gk}.\ref{thm:Gk:d}, $T_eG^{k+1}$ is included in $T_e(G^k)^1$, and by \cref{thm:Gk}.\ref{thm:Gk:b}, $R_X,R_Y$ are right-invariant $C^1$ vector fields on $G^k$. In particular, the expression $dR_Y(e)(X)-dR_X(e)(Y)$ is well defined in a chart around $e \in G^k$. The definition is independent of the chosen chart and therefore determines a unique element $[X,Y]\in T_eG^k$.
\item The tangent vector $[X,Y] \in T_eG^k$ defined in \ref{lem:lie_bracket:1} acts on functions $f$ as the derivation $f\mapsto R_XR_Yf-R_YR_Xf(e)$. Thus, the definitions in \ref{lem:lie_bracket:1} and \ref{lem:lie_bracket:2} coincide.
\item The vector field $R_X$ is right-invariant $C^1$ on $G^k$ by \ref{lem:lie_bracket:1}.
It  has a global $C^{1}$ flow $\on{Fl}^{R_X}:\mathbb R\times G^k \to G^k$. 
The flow $T\on{Fl^{R_X}_t}: TG^k\to TG^k$ is differentiable in $t$ at 0 since its velocity field $TX$ is Lipschitz 
by \ref{lem:right-invariance}.  
Therefore, the expression
\begin{equation}
(\on{Fl}^{R_X}_t)^* R_Y(e)
=
T(\on{Fl}^{R_X}_{-t}) \o R_Y \o \on{Fl}^{R_X}_t
\end{equation}
is differentiable in $t$ at $t=0$, and we have shown that $[X,Y]$ in \ref{lem:lie_bracket:3} is well defined.
This definition coincides with the one in \ref{lem:lie_bracket:2} by a well-known argument \cite[Lemma~32.15]{KrieglMichor97}:
for any smooth function $f$ defined near $e \in G^k$ and for sufficiently small $s,t\in\mathbb R$,
define
\begin{equation}
\al(t,s)=R_Y(f\o\on{Fl}^{R_X}_{-s})(\on{Fl}^{R_X}_t(e)).
\end{equation}
Then, $\partial_u\al(u,u)=[X,Y]f$ in the sense of \ref{lem:lie_bracket:2} because
\begin{align}
\al(t,0)
&=
R_Y(f)(\on{Fl}^{R_X}_t(e)),
&
\al(0,s)
&=
R_Y(f\o\on{Fl}^{R_X}_{-s})(e),
\\
\partial_t\al(0,0)
&=
R_XR_Yf(e),
&
\partial_s\al(0,0)
&=
-R_YR_Xf(e).
\end{align}
Moreover, one also has $\partial_u|_0\al(u,u)=[X,Y]f$ in the sense of \ref{lem:lie_bracket:3} because
\begin{align}
\al(u,u)
&=
(T\on{Fl}^{R_X}_{-u}R_Y)(f)(\on{Fl}^{R_X}_u(e))
=
((\on{Fl}^{R_X}_u)^*R_Y)(f)(e).
\end{align}
Thus, the definitions in \ref{lem:lie_bracket:2} and \ref{lem:lie_bracket:3} coincide.\qedhere
\end{enumerate}
\end{proof}

\section{Regularity of half-Lie groups}
\label{sec:regularity}

\begin{definition}[Regular half-Lie groups]
Let $G$ be a Banach right half-Lie group, and let $\mathcal F$ be a subset of $L^1_\loc(\mathbb R,T_eG)$.
Then, $G$ is called $\mathcal F$-regular if for all $X\in \mathcal F$, there exists a unique solution 
$g \in W^{1,1}_\loc(\mathbb R,G)$\footnote{The space $W^{1,1}_\loc(\mathbb R,G)$ consists of all $g$ such that $g$  and  $g' \in L^1_\loc$.} of the differential equation
\begin{equation}
\label{eq:evol}
\partial_t g(t)= T_e \mu^{g(t)}X(t),
\qquad
g(0)=e.
\end{equation}
This solution is called the \emph{evolution} of $X$. 
It will be denoted by $\on{Evol}(X)$, and its evaluation at $t=1$ by $\on{evol}(X)$.
In the special case $\mathcal F=C^{\infty}(\mathbb R,T_eG)$, we shall simply speak of regularity of $G$.
\end{definition}

Every Banach Lie group $G$, and in particular every finite dimensional Lie group, is regular.
Indeed, the time-dependent right-invariant vector field
\begin{equation}
R_X:\mathbb R\times G\ni (t,x) \mapsto T_e\mu^x(X(t)) \in T_xG
\end{equation}
is smooth, hence its integral curves exist uniquely for all time, and the evolution of $X$ is the integral curve of $R_X$ started at the identity element.
On half-Lie groups $G$, the situation is more subtle because $R_X$ may be non-smooth, even if $X$ is smooth.
However, if $X$ takes values in $T_eG^k$, then $R_X$ is a time-dependent $C^k$ vector field on $G$.
This is used in the following \lcnamecref{thm:Gkreg} to show regularity of $G^k$ for any $k\geq 1$.

\begin{theorem}[Regularity of $G^k$]
\label{thm:Gkreg}
Let $G$ be a Banach right half-Lie group carrying a right-invariant local addition.
Then, for any $k \in \mathbb N_{\geq 1}\cup{\infty}$, $G^{k}$ is regular.
\end{theorem}

\begin{proof}
We use superscripts $G$ to denote invariance with respect to the right-action of $G$. 
Let $X\in C^{\infty}(\mathbb R, T_e G^{k})$ for $k\in \mathbb N_{\geq 1}$.
Then, $X$ extends uniquely to a time-dependent right-invariant $C^k$ vector field $R_X \in C^{\infty}(\mathbb R, \mathfrak X_{C^k}(G)^G)$.
Consequently, $R_X$ has a global flow $\on{Fl}^{R_X}:\mathbb R\times G\to G$.
As $R_X$ is right-invariant and $C^k$, it follows for each $t \in \mathbb R$ that $\on{Fl}^{R_X}_t$ is right-invariant and $C^k$.
The curve $\mathbb R \ni t\mapsto \on{Fl}^{R_X}_t \in \Diff_{C^k}(G)^G$ is differentiable.
To see this, one works locally in a chart for $\Diff_{C^k}(G)^G$ as described in \cref{lem:CkGGG_manifold}, rewrites the flow equation in integral form, verifies using \cref{lem:Ck_calculus}.\ref{lem:Ck_calculus:4} that the integrand is a continuous function of time with values in the function space $C^k(G,G)$, and concludes that the left-hand side of the integral equation is continuously differentiable in $C^k(G,G)$.
Thus, the evolution of $R_X$ in $\Diff_{C^k}(G)^G$ exists and is given by $\on{Fl}^{R_X}$.
Using the diffeomorphism $\on{ev}_e:\Diff_{C^k}(G)^G\to G^k$ described in the proof of \cref{thm:Gk}, one obtains that the evolution of $X$ in $G^k$ exists and is given by $\on{Fl}^{R_X}(e)$.
\end{proof}

\section{Extensions}
\label{sec:extensions}

In this part we study extension theory in the category of right half-Lie groups, with morphisms defined as smooth group homomorphisms. 
For Lie groups, the (cohomological) extensions have been worked out by  O.\ Schreier \cite{Schreier26, Schreier26a},
R.\ Baer \cite{Baer34},
S.\ Eilenberg and S.\ MacLane \cite{EilenbergMacLane47},
G.\ Hochschild \cite{Hochschild51, Hochschild54},
and G.\ Hochschild and J.-P\ Serre \cite{HochschildSerre53a}. This theory has been extended to infinite dimensional Lie groups by Neeb \cite{Neeb02,Neeb04,Neeb07}. 
A comprehensive presentation of this theory for (finite dimensional) Lie groups can be found in \cite[Sections 15.11--15.27]{Michor08}, and we shall refer to this, but we translate everything to right actions, which are better suited for our goal. We refrain from pushing the extension theory to its cohomological description; this can done similarly as in Neeb \cite{Neeb07} by considering appropriate local $C^{0,\infty}$-cocycles.

\begin{definition}[Smooth extensions]
\label{def:smooth_extensions}
Let $N$ and $G$ be right half-Lie groups. A right half-Lie group $E$ is called an \emph{extension} of $G$ over $N$ if there is a short exact sequence of smooth group homomorphisms:
\begin{equation}
\xymatrix@C=4em{ e  \ar[r]  & N  \ar[r]^{{i}}  
   & E  \ar[r]^{{p}}   & G  \ar[r]    & e}\,.
\end{equation}
The extension is said to admit:
\begin{enumerate}[(a)]
\item \label{def:smooth_extensions:a} \emph{local sections}  if $p$ admits a local smooth section $s$ near $e$ (equivalently near
any point), and $i$ is initial \cite[Section 27.11]{KrieglMichor97}.
\item \label{def:smooth_extensions:b} \emph{local retractions} if $i$ admits a local smooth retraction $r$ near $e$ (equivalently
near any point), and $p$ is final \cite[Section 27.15]{KrieglMichor97}.
\end{enumerate}
The extension is called a \emph{smooth extension} if both \ref{def:smooth_extensions:a} and \ref{def:smooth_extensions:b} are valid.
\end{definition}

If $E$ is a Lie group, then the two conditions \ref{def:smooth_extensions:a} and \ref{def:smooth_extensions:b} are equivalent because $s(p(x)).i(r(x))=x \in E$. They imply that $E$ is locally diffeomorphic to $N\x G$ via $(r,p)$ with local inverse $(i\o\on{pr}_1).(s\o\on{pr}_2)$. 
Under  conditions \ref{def:smooth_extensions:a} and \ref{def:smooth_extensions:b} $E$ is locally diffeomorphic to $N\x G$ via $(r,p)$ with local inverse $(i\o\on{pr}_1).(s\o\on{pr}_2)$.
Not every smooth exact sequence of even Lie groups admits local sections as required above. Let, for example, $N$ be a closed linear subspace in a convenient vector space $E$ which is not a direct summand, and let $G$ be $E/N$. Then the tangent mapping at 0 of a local smooth splitting would make $N$ a direct summand.

For extensions of Lie groups conjugation by elements in $E$ is automatically smooth on $N$. This can fail in the context of extensions of half-Lie groups, which  leads us to the following definition:
\begin{definition}[Smooth conjugation]
\label{def:smooth_conjugation}
An extension of half-Lie groups $e\to N \xrightarrow{i} E \xrightarrow{p} G\to e$ is said to admit a smooth conjugation
if for any $x \in E$, conjugation $n\mapsto x\, i(n)\, x\i$ induces a smooth automorphisms of $N$.
\end{definition}

\begin{definition}[Equivalent extensions]
\label{def:extension:equivalent}
Two extensions are {\it topologically (smoothly) equivalent} if there exists a continuous (diffeomorphic) isomorphism $\ph$ fitting commutatively into the diagram
\begin{equation}
\xymatrix@C=4em{ e  \ar[r]  & N  \ar[r]^{{i}}  \ar@{=}[d]
   & E  \ar[r]^{{p}}  \ar[d]_{{\ph}}  & G  \ar[r]  \ar@{=}[d]   & e\\
 e  \ar[r]  & N  \ar[r]^{{i'}} & E'  \ar[r]^{{p'}} & G  \ar[r] & e. }
\end{equation}
\end{definition}
Note that if a  homomorphism  $\ph$  exists, then it is an isomorphism. 
We first consider the important special case of split extensions.
  
\begin{definition}[Split extensions]
An extension is called \emph{split} if there exists a smooth group homomorphism $s:G\to E$ which is a section of $p$. 
\end{definition}

A seemingly weaker condition is that the section $s$ is a group homomorphism and smooth near $e$; then $s$ is automatically smooth on all of $G$.
By the following \lcnamecref{lem:extensions:semidirect}, split extensions are equivalent to semidirect products.

\begin{lemma}[Split extensions and semidirect products]
\label{lem:extensions:semidirect}
Consider an extension of  half-Lie groups $e\to N \xrightarrow{i} E \xrightarrow{p} G\to e$ with smooth conjugation.
\begin{enumerate}[(a)]
\item Any smooth group homomorphism $s$, which is a section of $p$, determines a continuous right action $\rh:N\times G\to N$ with $n\mapsto \rh(n,x)$ smooth for all $x\in G$, via
$$
\rh:N\x G\to N, \qquad i\rh(m,x) = s(x^{-1}).i(m).s(x)
$$
\item Any such right action $\rh$ defines a half-Lie group and semidirect product $G\ltimes N$, which is the product manifold $G\times N$ with the group operations
$$
(x,m).(y,n)
=
(xy,\rho(m,y)n),
\qquad
(x,m)\i
=
(x\i,\rho(m\i,x\i))
$$
\item \label{lem:extensions:semidirect:a} The extension $E$ is topologically equivalent to the smooth split extension
\begin{equation}
\xymatrix@C=4em{
 e  \ar[r]  & N  \ar[r]^{(e,\Id_N)} & G\ltimes N  \ar[r]^{\on{pr}_1} & G  \ar[r] & e }
\end{equation}
via the continuous isomorphism $G\ltimes N\ni (x,n)\mapsto s(x).i(n)\in E$.
\end{enumerate}
\end{lemma}
Note that this implies that every semidirect product is a \emph{smooth} (split) extension.
\begin{proof}
\begin{enumerate}[(a), wide]
\item One easily verifies that $\rh$ is a right action, and the claimed continuity and smoothness properties follow from conditions \ref{def:smooth_extensions:a}--\ref{def:smooth_extensions:b} of \cref{def:smooth_extensions} and the smoothness of the conjugation.
\item One easily verifies that the group operations are well defined, and the continuity and smoothness properties follow from those of $\rh$.
\item One easily verifies that $G\ltimes N$ is a smooth extension of $G$ over $N$ where the group homomorphism is given by $s(x)=(x,e)$ and where $\on{pr}_N:G\ltimes N\to N$ plays the role of the local smooth retraction required by \cref{def:extension:equivalent}. The equivalence of extensions is described by the continuous isomorphism $G\ltimes N\ni (x,n)\mapsto s(x).i(n)\in E$ 
which is smooth in $x$ for fixed $n$, and is also smooth in $n$ for fixed $x$ since 
$n\mapsto ( \rh(n,x\i)).s(x) = s(x).i(n)$ is smooth. So $G\ltimes N\ni (x,n)\mapsto s(x).i(n)\in E$ is separately smooth and a group homomorphism, but we do not know that is smooth. It has the continuous inverse $(p,r):E\to G\times N$, where $r:E\to N$ is given by $i(r(x))=s(p(x)\i).x$. 
\qedhere
\end{enumerate}
\end{proof}

The following theorem explicitly characterizes differentiable elements in semidirect products. 
We have no corresponding characterization for general split extensions because to our knowledge, these are only topologically and not smoothly equivalent to semidirect products. 

\begin{theorem}[Differentiable elements in semidirect products]
\label{thm:semidirect}
Consider the semidirect product of Banach half-Lie groups
\begin{equation}
\xymatrix@C=3em{
 e  \ar[r]  & N  \ar[r]^{i} & G\ltimes N  \ar[r]^{p} & G  \ar[r] & e, }
\end{equation} 
for some right action $\rh$ as in \cref{lem:extensions:semidirect}  with right-invariant local additions, where $i=(e,\Id_N)$ and $p=\on{pr}_1$. 
Then, for any $k \in \mathbb N$, the set of $C^k$ elements in $G\ltimes N$ is a semidirect product of Banach half-Lie groups
\begin{equation}
\xymatrix@C=4em{
 e  \ar[r]  & N^{k,\rho}  \ar[r]^(.3){i} & (G\ltimes N)^k =G^k\ltimes N^{k,\rh} \ar[r]^(.7){p} & G^k  \ar[r] & e, }
\end{equation}
where
\begin{align}
N^{k,\rho}:=i\i\big((G\ltimes N)^k\big)=\left\{m\in N^k: \rho(m,y)n \text{ is } C^k\text{ in } (y,n)\in  G\ltimes N \right\} \subseteq N^k.
\end{align} 
\end{theorem}

Note that \cref{thm:semidirect} does not hold with Fr\'echet differentiability replaced by Gateaux differentiability, as in \cite{neeb2010differentiable}.
Indeed, the action on Gateaux-$C^k$ elements is discontinuous \cite[Proposition~9.7]{neeb2010differentiable}.
Therefore, the multiplication in the semidirect product is discontinuous, in contradiction to the definition of half-Lie groups.

\begin{proof}
One easily verifies that $(G\ltimes N)^k=G^k\ltimes N^{k,\rho}$ as sets. Recall from \cref{thm:Gk} that $G^k$ is a Banach right half-Lie group with right-invariant local addition, as is $(G\ltimes N)^k$. Hence, the projection $p:G^k\times N^{k,\rho}\to G^k$ is a splitting submersion between Banach manifolds, i.e., the differential of $p$ is surjective at every point, and its kernel is complemented by the tangent space of $G^k$. 
By the implicit function theorem \cite[Proposition~II.2.2]{lang1999fundamentals} $(G\ltimes N)^k$ is locally diffeomorphic to the product manifold $N^{k,\rho}\times G^k$ for a uniquely determined Banach manifold structure on $N^{k,\rho}=p\i(e)$. 
Moreover, the mappings $i$ and $p$ are smooth and admit smooth local sections and local retractions, respectively.
\end{proof}

Differentiable elements have been widely studied in representation theory.
Classically, one considers a representation of a finite-dimensional Lie group $G$ on an infinite-dimensional vector space $N$. 
Then, differentiable elements in $N$ can be characterized in terms of the Lie algebra, form a dense subspace, and carry a natural topology, namely, the one induced from $C^k(G,N)$ \cite{goodman1969analytic, michor1990moment, gaarding1947note}.
For infinite-dimensional Lie groups, several problems arise: 
there are several distinct notions of differentiability, 
the set of differentiable elements may be trivial, and there seems to be no `good' topology on $C^k(G,N)$ \cite{neeb2010differentiable}. 
For instance, the smooth compact-open topology on $C^k(G,N)$ is too coarse to ensure continuity of the action \cite{neeb2010differentiable}.
Our solution to these problems, obtained as a corollary of \cref{thm:semidirect}, is based on the following two observations: 
first, it suffices to consider the subspace of $G$-invariant functions in $C^k(G,N)$, and second, the action becomes continuous when this subspace is endowed with the Fr\'echet-$C^k$ topology.

\begin{corollary}[Group representations]
Let $\ell:G\to L(N)$ be a representation of a Banach half-Lie group $G$ on a Banach space $N$ such that  $\ell:G\times N\to N$ is continuous. Then, the set of all $n \in N$ such that $g\mapsto \ell(g\i)n$ is $C^k$ is a Banach space, whose norm is induced by the norm on the $k$-jets at $e\in G$ of the maps $g\mapsto \ell(g\i)n\in N$.
\end{corollary}

\begin{proof}
One obtains a right-action, which satisfies the conditions of \cref{thm:semidirect}, by defining
\begin{equation}
\rho:N\times G\to N,
\qquad
\rho(n,g)=\ell(g\i)(n).
\end{equation}
Then, $N^{k,\rho}$ is a Banach manifold, and
\begin{align}
N^{k,\rho}
&=
\left\{n\in N^k: \rho(n,g')n' \text{ is } C^k\text{ in } (n',g')\in N\times G \right\}
\\&=
\left\{n\in N^k: \ell(g\i)n \text{ is } C^k\text{ in } g\in G \right\}
\\&=
\left\{n\in N^k: \ell(g)n \text{ is } C^k\text{ in } g\in G \right\}\quad\text{ if }G\text{ is a Lie group. } \qedhere
\end{align}
\end{proof}

\begin{remark}[Describing general extensions]
\label{rem:describing_extensions}
In general, there does not exist a continuous section $s:G\to E$. Indeed, an extension is in particular a topological principal $N$-bundle $E\to G$, and thus existence of a continuous section $s$ would imply that the bundle is topologically trivial. If  the extension admits local sections as in \ref{def:smooth_extensions:a} of \ref{def:smooth_extensions}, then there exists a discontinuous section $s:G\to E$ with $s(e)=e$ which is smooth on an open $e$-neighborhood $U\subseteq G$. Without loss of generality, $s$ is also smooth on $U\i$ and $U.U$, thanks to $G$ being a topological group. The section $s$ induces mappings
\begin{align}
\label{equ:alpha_f}
\al&: G\to \on{Aut}(N), & \al^x(n) &= s(x)\i n s(x),\\
f&: G\x G\to N,         & f(x,y) &= s(xy)\i s(x) s(y).
\end{align}
Assuming smoothness of conjugation in the sense of \cref{def:smooth_conjugation} implies that $\al$ takes values in the group $\on{Aut}(N)$ of smooth group automorphisms of $N$. 
Moreover, the associativity of multiplication implies the following properties of $\al$ and $f$:
\begin{equation}\label{eq:properties}
\begin{aligned}
\al^x\o\al^y &= \on{conj}_{f(x,y)\i}\o \al^{yx},               \\
f(e,e) &= f(x,e) = f(e,y) = e,                     \\
e  &=f(xy,z)\i f(x,yz)f(y,z) \al^z(f(x,y)\i).
\end{aligned}
\end{equation}
Note that the first property in \eqref{eq:properties} means that $\al$ induces a group anti-homomorphism
$\bar \al: G\to \on{Aut}(N)/\on{Int}(N)$, where $\on{Int}(N)$ is the normal subgroup of all inner automorphisms in $\on{Aut}(N)$.
In terms of $(\al,f)$, the group structure on $E$ is given by 
\begin{equation}\label{group_multiplication}
\begin{aligned}
s(x)m.s(y)n &= s(x)s(y)s(y)\i m s(y)n = s(xy)f(x,y)\al^y(m)n, \\ 
(s(x)m)\i &= s(x\i)\al^{x\i}(m\i)f(x,x\i)\i %
\end{aligned} 
\end{equation}
Since $E$ is a right half-Lie group this implies that $(x,m)\mapsto f(x,y)\al^y(m)n\in N$ is smooth near $(e,e)$. As $m\mapsto \al^y(m)$ is smooth, we conclude that $x\mapsto f(x,y)$ is smooth on $U$ if $N$ is a Lie group, as shall be assumed in the sequel.
\end{remark}

This leads us to make the following definition:

\begin{definition}[Extension data]\label{def:extension_datum}
An \emph{extension datum} for a half-Lie group $G$ and a Lie group $N$ is 
\begin{enumerate}[(a)]
\item a pair of mappings $\al: G\to \on{Aut}(N)$ and $f: G\x G\to N$ satisfying the properties listed in~\eqref{eq:properties}, such that
\item there exists an open neighborhood $U$ of $e \in G$ such that $\al:U\times N\to N$ is smooth, $f:U\times U\to N$ is continuous, and $x\mapsto f(x,y)$ is smooth on $U$ for each fixed $y\in U$. 
\end{enumerate}
\end{definition} 

Using this notation we obtain the following characterization for smooth extensions:

\begin{theorem}[Smooth extensions]\label{thm:non-split}
Let $e\to N \xrightarrow{i} E \xrightarrow{p} G\to e$ be an extension of half-Lie groups that 
admits local sections as in \ref{def:smooth_extensions:a} of \cref{def:smooth_extensions} and smooth  conjugations as in \cref{def:smooth_conjugation}. Assume in addition that $N$ is a Lie group. 
\begin{enumerate}[(a)]  
\item Any (possibly discontinuous) section $s:G\to E$ of $p$ with $s(e)=e$, which is smooth on some neighborhood of $e$, defines an extension datum $(\al,f)$ via \eqref{equ:alpha_f}.

\item Any extension datum $(\al,f)$ defines a half-Lie group $E(\al,f)$, whose underlying set is the product $G\times N$, endowed with the group operations 
\begin{equation}
 \begin{aligned}
    (x,m).(y,n) &= (xy, f(x,y)\al^y(m)n),\\
    (x,m)\i&= (x\i, \al^{x\i}(m\i)f(x,x\i)\i),
   \end{aligned}
\end{equation}
and with the manifold structure extended by right translations from $U\times N$, where $U$ is as in \cref{def:extension_datum}. 

\item  The extension $E$ is topologically equivalent to the smooth extension 
\begin{equation}
\xymatrix@C=4em{
 e  \ar[r]  & N  \ar[r]^{(e,\Id_N)} & E(\al,f)  \ar[r]^{\on{pr}_1} & G  \ar[r] & e. }
\end{equation}
         
\item Two data $(\al,f)$ and $(\al_1,f_1)$ define equivalent smooth extensions if there exists a mapping $b: G\to N$ (smooth near $e$) such that
\begin{align}
\al_1^x &=\on{conj}_{b(x)\i} \o \al^x, \\
f_1(x,y) &= b(xy)\i f(x,y)\al^y(b(x)) b(y).
\end{align}
The induced smooth isomorphism $E(\al,f)\to E(\al_1,f_1)$ between the corresponding extensions is given by $(x,n)\mapsto (x,b(x)n)$.
\end{enumerate}
\end{theorem}

For a comprehensive description of cohomological interpretations of extension data $(\al,f)$ we refer to \cite[Sections 15.11--15.27]{Michor08} with the appropriate changes (from left actions to right actions, and with appropriate weakened smoothness assumptions near $e$). 
Note that for split extensions, one may choose $\al=\rh$ to be a right action and $f$ to be constant and equal to $e$, as seen by comparison with \cref{lem:extensions:semidirect}.

\begin{proof}
\begin{enumerate}[(a), wide]
\item has been shown in \cref{rem:describing_extensions}. 
\item We describe the manifold structure on $E(\al,f)$ in some more detail.
Let $\tilde U:=\on{pr}_1\i(U)$ with the topological and manifold structure from $U\x N$. By assumption, $\al:U\times N\to N$ is smooth, $f:U\x U\to N$ is continuous, and $x\mapsto f(x,y)$ is
smooth. The group multiplication on $E(\al,f)$, which is inspired by formula~\eqref{group_multiplication}, is continuous $\on{pr}_1\i(V)\x \on{pr}_1\i(W) \to \tilde U$ for all $e$-neighborhoods $V,W\subset U$ with $V.W\subseteq U$. 
Likewise, right translations are smooth 
$\mu^y: \on{pr}_1\i(V)\to \tilde U$ for all $y\in \on{pr}_1\i(W)$.
We then use the charts $(\tilde U.x,\mu^{x\i}:\tilde U.x\to \tilde U)$, indexed by $x \in E(\al,f)$, as
an atlas for $E(\al,f)$. The chart changes are
$\mu^{y\i}\o\mu^x=\mu^{x.y\i}: (\tilde U.x\cap \tilde U.y).x\i=\tilde U\cap
(\tilde U.y.x\i)\to \tilde U\cap (\tilde U.x.y\i)$, so they are smooth.
The resulting smooth manifold structure on $E(\al,f)$ has the property that multiplication and inversion are continuous, and that right translations are smooth.
Moreover, $E(\al,f)$ is Hausdorff: Either
$\on{pr}_1(x)=\on{pr}_1(y)$, and then
we can separate them already in one chart $x.\tilde U = \on{pr}_1\i(\on{pr}_1(x).U)$,
or we can separate them with open sets of the form $\on{pr}_1\i(U_1)$ and $\on{pr}_1\i(U_2)$. 
This proves that $E(\al,f)$ is indeed a half-Lie group. 
\item The projection $\on{pr}_1:E(\al,f)\to G$ and inclusion $(e,\Id_N):N\to E(\al,f)$ are smooth and admit smooth local sections and retractions as required in \cref{def:smooth_extensions}.
Thus, $E(\al,f)$ is a smooth extension of $G$ over $N$. The topological equivalence of extensions 
is again given by the mapping $E(\al,f)\ni (x,n) \mapsto s(x).i(n)\in E$, where $s$ is the global section smooth near $e$ which was used to contruct the extension datum in \ref{rem:describing_extensions}.  This is a group isomorphism by comparing  \eqref{group_multiplication} and the multiplication in $E(\al,f)$ and  is continuous near $(e,e)$ and separately smooth in $x$ for fixed $n$ and in $n$ for fixed $x$ since 
$s(x).i(n) = \al^{x\i}(n).s(x)$. We do not know whether it is smooth in general.

\item The remaining assertion follows easily.
\qedhere
\end{enumerate}
\end{proof}

Next, we characterize the differentiable elements for non-split extensions described by extension data. 

\begin{theorem}[Differentiable elements in extensions]
\label{thm:Ek}
Consider a smooth extension described by an extension datum $(\al,f)$ as in  \cref{def:extension_datum}:
\begin{equation}
\xymatrix@C=4em{
 e  \ar[r]  & N  \ar[r]^(.4){i} & E=E(\al,f)  \ar[r]^(.6){p} & G  \ar[r] & e, }
\end{equation} 
where $i=(e,\Id_N)$ and $p=\on{pr}_1$. 
Assume that $N$, $G$, and $E$ are Banach half-Lie groups with right-invariant local additions, that $N$ is a Lie group and that the local addition on $E$ respects $i(N)$. Let $k \in \mathbb N$.
Then the set  of $C^k$ elements in $E$ is a smooth extension
\begin{equation}
\xymatrix@C=4em{ 
e  \ar[r]  
& 
i\i(E^k) \ar[r]^{i}
& 
E^k \ar[r]^(.4){p}  
& 
E^k/(i(N)\cap E^k) \ar[r] 
& 
e,}
\end{equation}
where all spaces are Banach half-Lie groups with right-invariant local additions, and the local addition on $E^k$ respects $i(i\i(E^k))=i(N)\cap E^k$.
Moreover, the extension $E^k$ is described by the restriction of the extension datum $(\al,f)$ to the above spaces, and  
\begin{equation}
i\i(E^k) 
=
N^{k,\al} 
=
\{n \in N: \al(n,x) \text{ is $C^k$ in $x$}\}.
\end{equation}
\end{theorem} 

\begin{proof} 
We consider the subgroup $E^k$ of $C^k$-elements for $k\in \mathbb N_{>0}$. We first show that $p$ maps $E^k$ to $G^k$ and  $i\i(E^k)\subseteq N^k$. Therefore let $x\in E^k$, so $\mu_x:E\to E$ is $C^k$. Then $p\o \mu_x = \mu_{p(x)}\o p:E\to G$ is $C^k$ and since $p$ is final, $\mu_{p(x)}:G\to G$ is $C^k$. Thus $p(x)\in G^k$. 
 If $n\in i\i(E^k)\subset N$, i.e,  $i(n)\in E^k\cap i(N)$ then $i\o \mu_n = \mu_{i(n)}\o i: N\to E$ is $C^k$ and since $i$ is initial, $\mu_n:N\to N$ is $C^k$. Thus $i\i(E^k)\subseteq N^k$. 

Next we will show that 
\begin{align}
i\i(E^k)&= N^{k,\al} = \{m\in N: \al(m,\cdot):G\to N \text{ is }C^k\text{ near }e\in G\}
\\
E^k/(i(N)\cap E^k) &\subseteq G^{k,f} = \{x\in G^k: f(x,\cdot):G\to N \text{ is } C^k \text{ near }e\in G\}\,.
\end{align}
Note that, in general, $G^{k,f}$ does not seem to be a group. 

Namely, in terms of the extension datum $(\al,f)$, we have near $(e,e)$ where the extension is diffeomorphic to  $U^N\x U^G$, and using that $N$ is a Lie group,
\begin{align}
i(i\i(E^k)) &= (\{e\}\x N)\cap E^k 
\\&
=\{(e,m): (e,m)(y,n) = (y, f(e,y)\al(m,y)n) = (y,\al^y(m)n)\text{ is }C^k\text{ in }(y,n)\}
\\&
=\{m\in N: \al(m, \cdot): G\to N \text{ is }C^k\}\,.
\end{align}
Since a left translation commutes with all right translations, it is $C^k$ if and only if it is $C^k$ near the identity. 
\\
For the second assertion we look again at the multiplication near the identity: Since
\begin{align}
(x,m).(y,n) &= (xy, f(x,y)\al(m,y)n) \text{ is }C^k\text{ in }(y,n)
\end{align}
it follows that $x\in G^k$ and that $f(x,y)\al(m,y)n$ is $C^k$ in $(y,n)$ with values in the Lie group $N$. Thus $y\mapsto f(x,y)\al(m,y)$ is $C^k$ and since we already know that $\al(m,\cdot):G\to N$ is $C^k$ we conclude that $f(x,\cdot):G\to N$ is $C^k$. 

Next we study the subgroup $i(N)\cap E^k$ in terms of the diffeomorphism $E^k\xrightarrow{x\mapsto \mu_x} \Diff^k(E)^E$ from \cref{lem:right-invariance}. We aim to show that
\begin{align}
i(N)\cap E^k &= \{\ph\in \Diff^k(E)^E: \ph(i(N))=i(N)\}  %
\\&
= \{\ph\in \Diff^k(E)^E: T_e\ph(T_e(i(N)))\subseteq Ti(N) \}. 
\end{align}
To see this let $c:\mathbb R\to i(N)$ be a smooth curve then $\p_t(\ph(c(t)))= T\ph.c'(t) \in Ti(N)$ if and only if 
$Ti(N) \ni T\mu^{c(t)\i}.T\ph.c'(t) = T_e\ph.T\mu^{c(t)\i}.c'(t)$, by the right invariance of $\ph$. Thus $\p_t(p\o\ph\o c)(t) =0$ and $\ph$ maps $i(N)$ to $i(N)$. The rest is clear.

Thus we have shown that $i(N)\cap E^k$ is locally near $\Id_E$  diffeomorphic to an open set in the space of all right invariant $C^k$-vectorfields on $E$ which are tangent to $i(N)$, via the chart induced by a right invariant local addition $\ta^E$ respecting $i(N)$. This space is a closed linear subspace in the Banach space 
$\X_{C^k}(E)^E$ from \cref{lem:right-invariance} which has as complement the closed linear subspace of all $X\in \X_{C^k}(E)^E$ such that $X(e)$ is tangent to $T_es:T_eG \to T_eE$ where $s$ is a section of $p$ which is smooth near $e$.
Thus $E^k$ is locally diffeomorphic to the product manifold $N^{k,\alpha}\times E^k/N^{k,\alpha}$ for uniquely determined differentiable structures on  both factors. 
In particular the  conditions in \cref{def:smooth_extensions} hold.

The right invariant local addition $\ta:TE\supseteq V\to E$ maps $TE^k\cap V \to E^k$ smoothly by \cref{thm:Gk} and $Ti(N)\cap V\to i(N)$, thus also 
$T(i(N\cap E^k))\cap V \to i(N)\cap E^k$ smoothly. Thus, the induced local addition $\ta: TE^k\cap V\to E^k$ respects the normal subgroup $i(N)\cap E^k$. Since it is smooth on $E^k$,  it induces local additions on $i\i(E^k)$ and the quotient $E^k/N^{k,\al}$.  
\end{proof}

The above theorem supposes the existence of a right-invariant local addition respecting $i(N)$. 
This means that $i(N)$ is totally geodesic in $E$ with respect to the induced linear connection described in \cref{sec:local_additions}.
 The following result is included only to show how difficult it is to construct such a right invariant local addition on a general extension. 

\begin{proposition}[Local additions respecting a subgroup] 
Consider a smooth extension of half-Lie groups $e\to N \xrightarrow{i} E \xrightarrow{p} G\to e$, where $N$ is a Lie group.
Let $\ta^E:TE\supset V\to E$ be a local, right invariant addition that respects $i(N)$. 

Then $\ta^E$ induces  right invariant local additions $\ta^N:TN\supset i\i(V)\to N$ and $\ta^G:TG\supset p(V) \to G$ satisfying $p\o \ta^E = \ta^G\o Tp$ and $i\o \ta^N = \ta^N\o Ti$ and relating to the extension data $(\al,f)$ as follows:
\begin{align}
f(\ta^G(X_x),y)\al^y(\ta^N(Y_m)) &= \ta^N(T\mu^{\al^y(m)}.T_xf(\cdot,y).X_x + T\mu_{f(x,y)}.T_m\al^y.Y_m)\,.
\end{align}
\end{proposition}
Note the special case $f(\ta^G(X_x),y) = \ta^N(T_xf(\cdot,y).X_x)$. 

\begin{proof} In general, the tangent bundle of a half-Lie group is not a group since we cannot differentiate the multiplication $\mu$. But we get a sequence of vector bundles carrying right actions of the base half-Lie groups
\begin{equation}
\xymatrix{
& TN \ar[r]^{Ti} \ar[d]_{\pi_N} & TE \ar[r]^{Tp} \ar[d]_{\pi_E} & TG  \ar[d]_{\pi_G} & \\
e \ar[r] & N \ar[r]^i & E \ar[r]^p & G \ar[r] & e \\
}\label{2}
\end{equation}
which is fiberwise exact in the sense that $T_zp:T_zE \to T_{p(z)}G$ is onto with kernel $\on{ker}(T_zp) = T\mu^z(T_ei(T_eN))$ and $T_ni:T_nN\to T_{i(n)}E$ is injective with $\on{im}(T_ni) = \on{ker}(T_{i(n)}p)$ for all $z\in E$ and $n\in N$.
In terms of the extension data $(\al,f)$, differentiating the right translation on $E$  near the identity, where $E$ looks locally like $(G\x N)$, we get 
\begin{align}
(x,m).(y,n) &= (xy, f(x,y)\al^y(m)n) \label{3}
\\
T_{(x,m)}\mu^{(y,n)}(X_x,Y_m) &= \big( T_x\mu^y.X_x, T\mu^{\al^y(m)n}.T_xf(\cdot,y).X_x + T\mu_{f(x,y)}.T\mu^n.T_m\al^y.Y_m\big)\,.
\end{align}
Given $\ta^E$ we get $\ta^N$ by assumption and $\ta^G$ by \eqref{2} so that near $(e,e)$ we have $\ta^E(X_x,Y_m) = (\ta^G(X_x),\ta^N_m(Y,m))$. 
Since $\ta^E$ is right invariant, 
\begin{align}
\ta^E\big(T_{(x,m)}\mu^{(y,n)}(X_x,Y_m)\big) &= \mu^{(y,n)}\ta^E(X_x,Y_n) = (\ta^G(X_x),\ta^N(Y_m))(y,n) 
\\&
= \big(\ta^G(X_x).y,f(\ta^G(X_x),y)\al^y(\ta^N(Y_m)).n\big)
\end{align}
On the other hand, by using \eqref{3} we get 
\begin{align}
&\ta^E\big(T_{(x,m)}\mu^{(y,n)}(X_x,Y_m)\big)
\\&
=\ta^E\big( T_x\mu^y.X_x, T\mu^{\al^y(m)n}.T_xf(\cdot,y).X_x + T\mu_{f(x,y)}.T\mu^n.T_m\al^y.Y_m\big) 
\\&
=\big( \ta^G(T_x\mu^y.X_x), \ta^N(T\mu^{\al^y(m)n}.T_xf(\cdot,y).X_x + T\mu_{f(x,y)}.T\mu^n.T_m\al^y.Y_m)\big)
\end{align}
Comparing parts, we first see that $\ta^G$ and $\ta^N$ are right invariant as follows:
\begin{align}
\ta^G(X_x).y &=  \ta^G(T_x\mu^y.X_x), \text{ so }\ta^G\text{ is right invariant, and }
\\
f(\ta^G(X_x),y)\al^y(\ta^N(Y_m)).n &= \ta^N(T\mu^{\al^y(m)n}.T_xf(\cdot,y).X_x + T\mu_{f(x,y)}.T\mu^n.T_m\al^y.Y_m) \label{4}
\\
\ta^N(Y_m).n &= \ta^N(T\mu^n.Y_m), \text{ by choosing  }y=e \text{ above, and}  
\\
f(\ta^G(X_x),y).n &= \ta^N(T\mu^{n}.T_xf(\cdot,y).X_x) = \ta^N(T_xf(\cdot,y).X_x).n, 
\end{align}
where at the end we chose $m=e$ and $Y_m = 0_e$ and used right invariance of  $\ta^N$.
\end{proof}

\subsection{Central extensions}\label{central_extensions}
An extension $e\to N \xrightarrow{i} E \xrightarrow{p} G\to e$ is central if the Lie group $N$ is in the center of the right half-Lie group $E$. In terms of the extension datum $(\al,f)$ we have $\al^x= \on{Id}_N$, and $f:G\x G \to N$ is continuous near $(e,e)$ with 
$x\mapsto f(x,y)$ smooth near $e$ and, since $N$ is abelian, 
\[
f(e,e) = f(x,e) = f(e,y) = e,   \quad                 
f(y,z) f(xy,z)\i f(x,yz) f(x,y)\i=e;
\]
i.e., $f$ is a \emph{normalized group cocycle}. 
The central extension is then the half-Lie group $E$ given on the set $G\x N$ by 
\[
(x,m)(yn) = (xy, f(x,y)mn),\qquad (x,m)\i = (x\i, m\i f(x,x\i)\i),
\]
the topology and the smooth structure is extended from a neighborhood of $(e,e)$ by right translations.
The set of $C^k$-elements in $E$ is then the  central extension by $N$ described by the cocycle $f$ restricted to 
\[
G^{k,f} = \{ x\in G^k: f(x,\cdot):G\to N \text{ is }C^k\text{ near  }e\}\,,
\]
which in the central case is a group. Here $e\to N \xrightarrow{i} E^k \xrightarrow{p} G^{k,f}\to e$ is again a central extension.
An example is the central extension of $\Diff_{H^s}(S^1)$  or $\Diff_{H^s}(\mathbb R)$ over $S^1$ or $\mathbb R$, the Sobolev version of the Virasoro-Bott group.

\section{Examples of extensions}\label{sec:examples_extensions}

In the following we will present two important examples of extensions which even in the case of regular Lie groups seem to be new.

Following \cite{AbbatiCirelliManiaMichor86} we introduce the following spaces, which will be the building blocks of our first extension example. 

\begin{definition}\label{def:gauge}
For  a smooth principal bundle  $q:P\to M$ over a (for simplicity) compact Riemannian manifold $M$ with finite dimensional structure group $G$ with principal right action $\rh:P\x G\to P$ we let
\begin{align}
&\on{Aut}_{C^k}(P)=\left\{f\in\Diff_{C^k}(P): \rh^g\o f=f\o \rh^g \text{ for all }g\in G \right\},\\
&\on{Gau}_{C^k}(P)=\left\{f\in \on{Aut}_{C^k}(P): p(f)=\on{Id}_M\in\Diff_{C^k}(M) \right\},
\end{align}
where $p:\on{Aut}_{C^k}(P)\to \Diff_{C^k}(M)$ is defined by $q\o f= p(f)\o q$. This is well-defined as  $q\o f:P\to P\to M$ is constant on the fibers and thus factors in the above way. 
For $k=\infty$ we will also write $\on{Aut}(P)=\on{Aut}_{C^{\infty}}(P)$ and $\on{Gau}(P)=\on{Gau}_{C^{\infty}}(P)$ and for $l>\dim(P)/2+1$ we will also consider the $H^l$ versions $\on{Aut}_{H^l}(P)$ and $\on{Gau}_{C^{l}}(P)$.
\end{definition}
Note, that the mapping  $p$ is a group homomorphism onto an open normal subgroup of $\Diff(M)$, which we will denote by $\Diff^P(M)$. 
Namely, the action of $\Diff/\Diff^0$ (where $\Diff^0$ denotes the connected component) on the moduli space of $G$-principal bundles over $M$ might be non-trivial. Note that $\Diff^P(M)$ can be made into $\Diff(M)$ by replacing $P$ by the disjoint union of all these bundles below.
Furthermore, the gauge group $\on{Gau}(P)$ is isomorphic to the group $C^{\infty}(P, (G,\on{conj}))^G$ of smooth maps $h:P\to G$ which are equivariant in the sense that $h\o \rh^g = \on{conj}_g\o h$; this in turn is isomorphic to the space $\Ga(P[G,\on{conj}])$ of smooth sections of the associated group bundle $P[G,\on{conj}] \to M$; see \cite[Section 18.15]{Michor08}.  

\begin{theorem}[Automorphism groups of principal fiber bundles]\label{thm:gaugeextension}
In the setting of \cref{def:gauge}  the following is a smooth extension of regular Lie groups:
\begin{equation}
\xymatrix{
\{\on{Id}_P\} \ar[r] & \on{Gau}(P) \ar[r]^{i=\on{incl}} & \on{Aut}(P) \ar[r]^p & \Diff^P(M) \ar[r] & \{\on{Id}_M\}
}
\end{equation}
For the corresponding counterpart of mappings with finite regularity we obtain extensions of half-Lie groups admitting  local sections as in \ref{def:smooth_extensions:a} of \cref{def:smooth_extensions} and smooth  conjugations as in \cref{def:smooth_conjugation}:
\begin{gather}
\xymatrix{
\{\on{Id}_P\} \ar[r] & \on{Gau}_{C^k}(P) \ar[r]^{i=\on{incl}} & \on{Aut}_{C^k}(P) \ar[r]^p & \Diff^P_{C^k}(M) \ar[r] & \{\on{Id}_M\}
}
\\
\xymatrix{
\{\on{Id}_P\} \ar[r] & \on{Gau}_{H^k}(P) \ar[r]^{i=\on{incl}} & \on{Aut}_{H^k}(P) \ar[r]^p & \Diff^P_{H^k}(M) \ar[r] & \{\on{Id}_M\}
}
\end{gather}
\end{theorem}

The smooth case for even infinite dimensional srtucture groups is also treated in \cite{Wockel07}.

\begin{proof}
We will first show the result in the smooth category. 
That $\on{Aut}(P)$ and $\on{Gau}(P)$ are smooth manifolds and regular Lie groups can be shown similarly as in the proof of the more complicated case in~\cref{thm:fiberdiffextension}. 
It remains to show  existence of local smooth sections of $p$ and retractions of $i$. For that we use a principal connection $\om\in \Om^1(P,\mathfrak g)$ on the bundle $P\to M$ and the corresponding  horizontal lift mapping $C:P\x _M TM\to TP$ given by $C(u,\cdot):= (T_uq|_{\ker \om_u})\i: T_{q(u)}M \to \ker \om_u\subset T_uP $ which satisfies $C(u.g, X_{x}) = T_u r^g. C(u, X_x)$ for $x=q(u)$ and $X_x\in T_xM$. Note that for any smooth vector field $X\in \X(M)$ on $M$ its horizontal lift $C(X)\in \X(P)$ which is given by $C(X)(u)= C(u,X(q(u)))$ is horizontal (i.e., in the kernel of $\om$) and is invariant under the principal right action, $(r^g)^* C(X)= X$. This also holds for non-autonomous vector fields on $M$.

We now choose a Riemannian metric on $M$ and use its exponential mapping $\exp:TM\to M$, which gives us a diffeomorphism $TM\supseteq V \xrightarrow[]{\pi_M,\exp} U\subseteq M\x M$ from an open neighbourhood $V$ of the 0-section onto an open neighborhood $U$ of the diagonal. 
For a diffeomorphism $\ph\in\Diff^P(M)$ which is so near to $\on{Id}_M$ that $(x,\ph(x))\in U$ for all $x\in M$, we consider the smooth curve $[0,1]\ni t\mapsto \ph_t\in\Diff^P(M)$ given by  $\ph_t(x)= \exp(t.(\pi_M,\exp)\i (\ph(x)))$ from $\on{Id}_M$ to $\ph$ and its right logarithmic derivative $t\mapsto X_t$ given by $X_t(\ph_t\i(x))= \p_t\ph_t(x)$ or $X_t= (\p_t\ph_t)\o \ph_t\i$.  The evolution (integral curve up to time 1) of the horizontal lift $f=\on{evol}(t\mapsto C(X_t))$ is then a $G$-equivariant smooth diffeomorphism of $P$ with $p(f)= \ph$ which also depends smoothly on $\ph$ and thus gives us the required smooth local section $s$ of $p$. The corresponding smooth local  retraction $r:\on{Aut}(P)\to \on{Gau}(P)$ is then given by $i(r(f))= f\o s(p(f))\i$. This proves that we have a smooth extension.
The inclusion $i:\on{Gau}(P) \to \on{Aut}(P)$ is initial because $i$ is the embedding of a closed smooth submanifold. This can be seen by adding $G$-equivariance to the proof of \cref{thm:fiberdiffextension}.

Next we prove the result for the case of finite regularity spaces. 
To see that $\on{Gau}_{C^k}(P)\cong \Ga_{C^k}(P[G,\on{conj}])$ and $\on{Gau}_{H^k}(P)\cong \Ga_{H^k}(P[G,\on{conj}])$ are actually Lie groups we note that the group structure is pointwise: the smooth fiberwise group multiplication and inversion on the finite-dimensional group bundle $P[G,\on{conj}]\to M$ acts by composition from the left on the space of $C^k$- or $H^k$-sections. 

Conjugation 
in $\on{Aut}_{C^k}(P)$ induces smooth automorphisms of $\Ga_{C^k}(P[G,\on{conj}])$ and similarly for the $H^k$-case:
To see this we use the smooth mapping $\ta:P\x_M P\to G$  which is uniquely given by $u_x.\ta(u_x,v_x)= v_x$ and satisfies 
$\ta(u_x.g, v_x.g')= g\i.\ta(u_x,v_x).g'$ and $\ta(u_x,u_x)=e$, where $u_x.g=\rh^g(u_x)$ for short; see \cite[Sections 18.2 and 18.15]{Michor08}. For $f\in \on{Aut}_{C^k}(P)$ with $f(u_x)=:v_y$ and $h\in \on{Gau}_{C^k}(P)$ we have 
\begin{align}
(f\o h\o f\i)(v_y) &= (f\o h)(f\i(v_y)) = f(h(u_x)) = f(u_x.\ta(u_x, h(u_x))) 
\\&
= f(u_x).\ta(u_x,h(u_x)) = v_y.\ta(u_x,h(u_x))
\end{align}
Since $P\ni u_x\mapsto \ta(u_x,h(u_x))$ is the representative of $\on{Gau}_{C^k}(P)\ni h$ in $C^k(P, (G,\on{conj}))$ by  
\cite[Section 18.15]{Michor08}, the result follows. 

Local smooth sections of $p:\on{Aut}_{C^k}(P) \to \Diff^P_{C^k}(M)$ can be constructed as in the smooth category, since the horizontal lift of a smooth connection on $P$ is smooth. Moreover, $i$ is initial, again for the same argument as in the smooth category. 
Similarly for $H^k$ instead of $C^k$.
\end{proof}

\begin{remark}
Note, that  $\on{Aut}(P)$ is a split smooth submanifold of $\Diff(P)$ if $G$ is compact, see the proof of \cref{thm:fiberdiffextension}. 
Furthermore, if the principal bundle $P\to M$ admits a flat connection, then the construction of $s$ given above yields a smooth group homomorphism up to the action of a discrete subgroup of holonomy transformations. This can be seen as follows. For a flat connection, the horizontal bundle is integrable and thus through each point $u\in P$ there exists a horizontal leaf $L$ such that $q|_L: L\to M$ is a covering map. The deck transformations of this covering map, extended $G$-equivariantly, form this discrete subgroup.
If $M$ is furthermore simply connected, then $q:L\to M$ is a diffeomorphism, and $s(\ph)(u.g) = (q|_L)\i(\ph(p(u))).g$ is a smooth group homomorphism.  In this case the above extensions reduce to direct products.
\end{remark}

Next we will study the situation for general fiber bundles. We introduce the following spaces of diffeomorphisms; 
In the following we will restrict ourselves to the connected components $\Diff^0$ of all diffeomorphism groups.
\begin{definition}\label{def:fiberdiff}
For a finite dimensional, compact, fiber bundle $q:E\to M$ we let
\begin{align}
\Diff^0_{C^k,\text{fiber}}(E):=\left\{f\in \Diff^0_{C^k}(E): q\o f = p(f)\o q \right\},\\
\Diff^0_{C^k,\text{fiber}}(E)_{\on{Id}_M}:=\left\{f\in \Diff^0_{C^k,\text{fiber}}(E):  p(f)=\on{Id}_M \right\}
\end{align}
where  $p:\Diff^0_{C^k,\text{fiber}}(E)\to \Diff^0_{C^k}(M)$ is uniquely determined by the defining relation $q\o f=p(f)\o q$. 
For $k=\infty$ we will also write $\Diff^0_{C^\infty,\text{fiber}}(E)=\Diff^0_{\text{fiber}}(E)$ and $\Diff^0_{C^\infty,\text{fiber}}(E)_{\on{Id}_M}=\Diff^0_{\text{fiber}}(E)_{\on{Id}_M}$ and for $l>\dim(P)/2+1$ we will also consider the $H^l$ versions  $\Diff^0_{H^l,\text{fiber}}(E)$ and $\Diff^0_{H^l,\text{fiber}}(E)_{\on{Id}_M}$.
\end{definition}
Then $p$ turns out to be a surjective group homomorphism, which can be seen similarly as in the proof of \cref{thm:gaugeextension}. 
We assume that the total space $E$ of the fiber bundle is compact to simplify the exposition. The constructions can be done also for non-compact $E$ and $M$ with more technical effort (all diffeomorphisms have to fall to the identity suitably near infinity).

\begin{theorem}[Fiber preserving diffeomorphisms of a compact fiber bundle]
\label{thm:fiberdiffextension}
In the setting of~\cref{def:fiberdiff}, the following is a smooth extension of Lie groups:
\begin{equation}
\xymatrix{
\{\on{Id}_E\} \ar[r] &\Diff^0_{\text{fiber}}(E)_{\on{Id}_M} \ar[r]^{i=\on{incl}} & \Diff^0_{\text{fiber}}(E) \ar[r]^{p} & \Diff^0(M) \ar[r] & \{\on{Id}_M\}
}
\end{equation}
For the corresponding counterpart of mappings with finite regularity we obtain extensions of half-Lie groups admitting  local sections as in \ref{def:smooth_extensions:a} of \cref{def:smooth_extensions}:
\begin{equation}
\xymatrix{
\{\on{Id}_E\} \ar[r] &\Diff^0_{C^k,\text{fiber}}(E)_{\on{Id}_M} \ar[r]^{i=\on{incl}} & \Diff^0_{C^k,\text{fiber}}(E) \ar[r]^{p} & \Diff^0_{C^k,}(M) \ar[r] & \{\on{Id}_M\}
\\
\{\on{Id}_E\} \ar[r] &\Diff^0_{H^k,\text{fiber}}(E)_{\on{Id}_M} \ar[r]^{i=\on{incl}} & \Diff^0_{H^k,\text{fiber}}(E) \ar[r]^{p} & \Diff^0_{H^k,}(M) \ar[r] & \{\on{Id}_M\}
}
\end{equation}
In contrast to \cref{thm:gaugeextension}, the kernels $\Diff^0_{C^k,\text{fiber}}(E)_{\on{Id}_M}$ and $\Diff^0_{H^k,\text{fiber}}(E)_{\on{Id}_M}$
are only half-Lie groups, and conjugation is only continuous.
\end{theorem}

\begin{proof} 
We first show that
the space $\Diff^0_{\text{fiber}}(E)$ is a regular Lie group and a closed subgroup of $\Diff^0(E)$. Therefore  we start with 
$C^{\infty}_{\text{fiber}}(E,E)$ where $\Diff^0_{\text{fiber}}(E)$ is open.
We use a spray $S:TE\to T^2E$ which is tangent to the fibers $TE_x$ and $q$-projectable to a spray $\bar S:TM\to T^2M$; this exists by a slightly elaborated version of  \cite[Section 5.9]{Michor20} which will appear in the future second edition of \cite{KrieglMichor97}.  We use their induced exponential mappings  which satisfy
\begin{gather}
q\o \exp^E = q\o \pi_E\o \on{Fl}^S_1 = \pi_M\o Tq \o \on{Fl}^S_1 = \pi_M\o \on{Fl}^{\bar S}_1 \o Tq = \exp^M \o Tq\,.
\\
TE\supset U^E\xrightarrow[]{(\pi_E,\exp^E)}  V^E\subset E\x E
\text{ and }
TM\supset U^M\xrightarrow[]{(\pi_M,\exp^M)}  V^M\subset M\x M
\end{gather}
are diffeomorphisms between neighborhoods of the zero sections and the diagonals which induce smooth charts taking values in the modelling spaces $\Ga(f^*TE)=C^{\infty}_f(E,TE)=\{s\in C^{\infty}(E,TE): \pi_E\o s=f\}$. 
\begin{align}
C^{\infty}(E,E)&\supset U_f=\{g\in C^{\infty}(E,E): (f,g)(E)\subset V^E\} \xrightarrow[]{u_f} \tilde U_f\subset \Ga(f^*TE)
\\
u_f(g) &=  (\pi_E,\exp^E)\i\o (f,g), \qquad  u_f(g)(x) = (\exp^E)_{f(x)}\i (g(x)),
\\
u_f\i (s) &= \exp^E\o s, \qquad (u_f\i(s))(x) = \exp^E_{f(x)}(s(x)), \text{ and }
\\
C^{\infty}(M,M)&\supset U_{\bar f}=\{\bar g\in C^{\infty}(M,M): (\bar f,\bar g)(M)\subset V^M\} \xrightarrow[]{u_{\bar f}} \tilde U_{\bar f}\subset \Ga(\bar f^*TM)
\\
u_{\bar f}(\bar g) &=  (\pi_M,\exp^M)\i\o (\bar f,\bar g),\qquad  u_{\bar f}\i (\bar s) = \exp^M\o \bar s
\end{align}
Note that the chart changes are just compositions from the left (push forwards) by smooth mapping like 
$$C^\infty_{f'}(E,TE)\ni s\mapsto (u_{f}\o u_{f'}\i )(s) = (\pi_E,\exp^E)\i \o (f, \exp^E\o s)\in C^\infty_{f}(E,TE).$$
If $f,g\in C^{\infty}_{\text{fiber}}(E,E)$ with $\bar f= p(f), \bar g= p(g) \in \Diff^0(M)$ then 
\begin{align}
Tq\o u_f(g) &= Tq\o (\pi_E,\exp^E)\i\o (f,g) = (\pi_M,\exp_M)\i\o (q\x q)\o (f,g) 
\\&
= (\pi_M,\exp_M)\i\o (q\o f, q\o g) = (\pi_M,\exp_M)\i\o (\bar f,\bar g)\o q = u_{\bar f}(\bar g)\o q
\end{align}
so that $u_f(g)\in C^{\infty}_f(E,TE)$ is $q$-projectable to $u_{\bar f}(\bar g) \in C^{\infty}_{\bar f}(M,TM)$. Since $Tq:TE \to TM$ is fiber linear over $q:E\to M$, the space  of  $q$-projectable vector fields along $f$ is a closed linear subspace of the space $C^{\infty}_f(E,TE)$ of all vector fields along $f$. 

It is even a complemented linear subspace: To see this, we choose a smoothly varying family of probability densities $\mu_x$ on the fibers $E_x$ over $x \in M$. For any $s \in C^\infty_f(E,TE)$, we define $\bar s \in C^\infty_{\bar f}(M,TM)$ by averaging over the fibers:
\begin{equation}
\bar s(x) = \int_{E_x} T_{f(y)}q.s(y) \mu_x(dy) \in T_{\bar f(x)}M, 
\qquad
x \in M. 
\end{equation}
Next, we choose an Ehresmann connection, determined by a horizontal bundle $HE$ in $TE$ complementary to the vertical bundle $VE= \ker Tq$, and define the linear map
\begin{equation}
C^\infty_f(E,TE) \ni s \mapsto s^{\on{ver}} + C(\bar s\o q,f) \in C^\infty_f(E,TE),
\end{equation}
where $s^{\on{ver}}$ is the vertical component of $s$, and $C:TM\x_M E\to TE$ is the horizontal lift.
This linear map retracts the set of all vector fields along $f$ onto the set of all $q$-projectable vector fields along $f$. 

Thus $C^{\infty}_{\text{fiber}}(E,E)$ is a split closed smooth submanifold of $C^\infty(E,E)$; consequently, $\Diff^0_{\text{fiber}}(E)$ is a splitting smooth submanifold of $\Diff^0(E)$ and the group operations are smooth as restrictions. The Lie algebra of $\Diff^0_{\text{fiber}}(E)$ is the space of all projectable smooth vector fields. 
The evolution of a smooth curve of projectable vector fields furnishes a smooth curve of projectable diffeomorphism, thus we get a regular Lie group. These arguments carry over to the group $\Diff^0_{\text{fiber}}(E)_{\on{Id}}$ of fiber preserving diffeomorphisms, which is again a smooth split submanifold of both $\Diff^0_{\text{fiber}}(E)$ and $\Diff^0(E)$. Consequently, 
$i:\Diff^0_{\text{fiber}}(E)_{\on{Id}_M}\to \Diff^0_{\text{fiber}}(E)$ is initial.

We can construct local smooth sections $s$ of $p:\Diff^0_{\text{fiber}}(E)\to \Diff^0(M)$ with the help of the smooth horizontal lift 
$C:TM \x_M E\to TE$ of an Ehresmann connection, see \cite[Section 17.3]{Michor08}, similarly as in the proof of  \cref{thm:gaugeextension}.

It remains to show the result for diffeomorphisms of finite regularity: note that the construction of charts, centered only at $C^\infty$-diffeomorphisms $f$, carries over to the situation here, i.e., to manifolds modeled on the Banach spaces of $q$-projectable vector fields in $\Ga_{C^k}(f^*TE)=C^k_f(E,TE)$ and $\Ga_{H^k}(f^*TE)=H^k_f(E,TE)$ and those which $q$-project to the identity, respectively. The chart changes described there are smooth, since they are push-forwards by the same smooth mappings. 

There are several ways to show that right translations are smooth: Either we consider charts centered at $C^k$- or $H^k$-diffeomorphisms, verify that chart changes are smooth, and note as in the proof of \cref{lem:CkGGG_manifold} that a right translation by $f$ is the identity in the charts $C^k_h(E,TE) \to C^k_{h\o f}(E,TE)$; similarly for $H^k$. 
Alternatively, we use only charts centered at smooth maps. Then, right translations are given by right compositions by functions of $C^k$ or $H^k$ regularity and left compositions by smooth functions; these are again smooth \cite[Lemmas 5.4 and 5.7]{Michor20}.
Yet another alternative is to use that right translations on $\Diff^0(E)$ are smooth and that the inclusion of $\Diff^0_{\text{fiber}}(E)$ in $\Diff^0(E)$ is smooth and initial.

Either way, this shows that all the involved spaces are half-Lie groups. The mappings $i$ and $p$ are smooth homomorphisms of groups since they are again bounded linear between suitably chosen charts.
Finally, smooth sections $s$ of $p$ can again be constructed with the help of an Ehresmann connection, as explicitly done in the proof of \cref{thm:gaugeextension}.  
\end{proof}

\section{Riemannian geometry on half-Lie groups}
\label{sec:riemannian}

This section investigates at various levels the links between the metric and manifold topologies for weak and strong Riemannian metrics, culminating in a Hopf--Rinow theorem for strong Riemannian metrics on half-Lie groups. 

\begin{definition}[Riemannian metrics on half-Lie groups]
Let $M$ be an infinite dimensional manifold, that is equipped with a smooth Riemannian metric $g$.
The metric is called \emph{right-invariant} if
\begin{equation}
g_x(T_e\mu^x h,T_e\mu^x k)=g_e (h,k),
\qquad
x \in G,
\qquad
h,k\in T_eG.
\end{equation}
The metric $g$ is called \emph{strong}
if the inner product $g_x$ induces the manifold topology on $T_xM$ for each $x\in M$; otherwise it is called \emph{weak}.
\end{definition}

Note that any right-invariant Riemannian metric $g$ is uniquely determined by the inner product $g_e$ on $T_eG$.
However, a given inner product on $T_eG$ does not necessarily extend to a \emph{smooth} right-invariant Riemannian metric on $G$ because $x\mapsto T_e\mu^x$ may be non-smooth and even discontinuous.
Nevertheless, many important half-Lie groups admit smooth right-invariant Riemannian metrics.

The following theorem states that any manifold carrying a strong Riemannian metric is a Hilbert manifold (i.e., modeled on a Hilbert space). 
The only assumption is that the manifold is modeled on a convenient vector space, i.e., a Mackey complete locally convex space. 
This is a mild completeness condition, and all Banach and Fr\'echet spaces are convenient.
The theorem is adapted from unpublished notes of Martins Bruveris and appears in a similar form without proof in \cite{AMT12}. 

\begin{theorem}[Strong Riemannian metrics]
\label{thm:riemannian:strong}
Let $g$ be a weak Riemannian metric on a convenient manifold $M$.
Then the following are equivalent:
\begin{enumerate}[(a)]
\item\label{strongmetric_a} $g$ is a strong Riemannian metric on $M$.
\item\label{strongmetric_b} $M$ is a Hilbert manifold and $g^\vee\colon TM\to T^*M$ is surjective.
\item\label{strongmetric_c} $M$ is a Hilbert manifold and $g^\vee\colon TM\to T^*M$ is a vector bundle isomorphism.
\end{enumerate}
\end{theorem}

\begin{proof}
\begin{enumerate}[wide]
\item[\ref{strongmetric_a}  $\Rightarrow$ \ref{strongmetric_b}:]
Fix a point $x \in M$.
By assumption, the locally convex topology of $T_xM$ is normable by $\|\cdot\|_{g_x}$.
In a normed space every Cauchy sequence is Mackey--Cauchy, and in a convenient vector space every Mackey--Cauchy sequence is convergent. Thus, $(T_xM,g_x)$ is a Hilbert space.
This implies that $M$ is a Hilbert manifold because $M$ is locally diffeomorphic to $T_xM$, as is easily seen in a chart.
Moreover, $g^\vee\colon TM\to T^*M$ is surjective by the Riesz representation theorem.

\item[\ref{strongmetric_b} $\Rightarrow$ \ref{strongmetric_c}] By the open mapping theorem, $g_x^\vee\colon T_xM\to T_x^*M$ is a linear  isomorphism for each $x\in M$. 

\item[\ref{strongmetric_c}  $\Rightarrow$ \ref{strongmetric_a}:]
Fix a point $x \in M$
and define $U = \{ u\in T_xM : \|u\|_{g_x}  < 1\}$.
Then $U$ is absolutely convex.
As $g_x^\vee$ is surjective, every bounded linear functional on $T_xM$ is of the form $g^\vee(v)$ for some $v \in T_xM$.
Every such functional is bounded on $U$ because $\sup_{u\in U}|\langle g^\vee(v), u \rangle| = |g (v,u)| \leq \|v\|_g $.
Hence $U$ is bounded in $T_xM$.
The inner product $g_x$ is bounded bilinear, which is equivalent to continuous bilinear because $T_xM$ is metrizable.
Thus, the norm $\|\cdot\|_{g_x}$ is continuous, and consequently $U$ is open in $T_xM$.
Now $U$ is an absolutely convex and bounded $0$-neighborhood in $T_xM$.
Then the Minkowski functional $v\mapsto \inf \{ t > 0: v \in tU \}$ is a norm which induces the locally convex topology of $T_xM$ \cite[Proposition~6.8.4]{jarchow2012locally}.
As $\mu_U = \|\cdot\|_{g_x}$, the topology induced by $g$ coincides with the topology of $T_xM$.
This holds for all $x \in M$, and therefore $g$ is a strong Riemannian metric.\qedhere
\end{enumerate}
\end{proof}

\subsection*{Geodesic distance.} For a weak Riemannian metric, the topology induced by the metric on a tangent space is weaker than the manifold topology. 
In this sense, the link between the metric and manifold topology is broken at an infinitesimal level.
At a global level, this results in the surprising phenomenon that the geodesic distance may fail to separate points or may even vanish completely.
Examples of such weak Riemannian metrics have been 
found first by Eliashberg and Polterovich~\cite{eliashberg1993bi} on the symplectomorphism group and then by Michor and Mumford on the full diffeomorphism group and on spaces of immersions~\cite{michor2005vanishing}. 
See  also \cite{bauer2012vanishingKdV, jerrard2019geodesic}.
In the following theorem, which appeared first in~\cite{bauer2020vanishing}, we characterize the vanishing-geodesic-distance phenomenon for Banach half-Lie groups:

\begin{theorem}[Vanishing geodesic distance]
Let $G$ be a Banach half-Lie group with a right-invariant Riemannian metric.
Assume that left-translation by any $x \in G$ is Lipschitz continuous with respect to the geodesic distance $d$, i.e.,
\begin{equation}\label{ass:lipschitz}
|\mu_x| := \inf\left\{C \in \mathbb R_+: d(xx_0,xx_1) \leq C d(x_0,x_1),  \forall x_0,x_1 \in G \right\} <\infty\;.
\end{equation}
Then the group elements with vanishing geodesic distance to the identity form a normal subgroup.
If the Riemannian metric is strong, then the geodesic distance is always non-degenerate, i.e., the normal subgroup of elements with vanishing geodesic distance
is the trivial subgroup.
\end{theorem}

\begin{proof}
Let $G^0$ be the set of all group elements $x\in G$ with vanishing geodesic distance to the identity, i.e., $d(e,x)=0$.
Then $G^0$ is a subgroup of $G$ because it holds for each $x_0, x_1 \in G^0$ that
\[
d(e,x_0x_1^{-1})\leq d(e,x_1^{-1})+d(x_1^{-1},x_0x_1^{-1})=d(x_1,e)+d(e,x_0)=0,
\]
where we have used the triangle inequality and right-invariance of the geodesic distance.
Moreover, $G^0$ is a normal subgroup of $G$ because it holds for all $x_0 \in G^0$ and $x \in G$ that
\[
d(e,xx_0x^{-1})=d(x,xx_0)\leq |\mu_x| d(e,x_0)=0,
\]
where we have used the right-invariance of the geodesic distance and the Lipschitz property of left-translations.
The non-degeneracy for strong metrics follows by the same arguments as in finite dimensions, see eg.~\cite{lang1999fundamentals}.
\end{proof}

\subsection*{Geodesic equation.}
The geodesic equation on Lie groups can be expressed equivalently in Lagrangian form as $\smash{\nabla_{\partial t}c_t=0}$ or in Eulerian form as $\smash{u_t=\on{ad}_u^\top u}$, where $\smash{u=(T_e\mu^c)\i c_t}$ is the right-trivialized velocity \cite{arnold1966, arnold1998topological}.
This is thanks to the formula $\smash{\nabla_{R_X}R_X=R_{\on{ad}_X^\top X}}$, which relates the Levi-Civita covariant derivative  to the transpose of the adjoint; here $R_X$ is the right invariant vector field corresponding to $X$
Existence of one implies existence of the other and always holds for strong Riemannian metrics but not necessarily for weak Riemannian metrics~\cite{bauer2014homogenous}. 
For half-Lie groups, the Lagrangian description remains valid, but the Eulerian description breaks down. 
There are several reasons for this. 
First, the right-trivialized velocity $u$ may not be differentiable (or even continuous) in time and therefore cannot solve the Eulerian geodesic equation $\smash{u_t=\on{ad}_u^\top u}$ in any standard sense.
Second, as we show next, the transposed adjoint $\smash{\on{ad}_u^\top u}$ exists merely for $u \in T_eG^1$ and is non-unique unless $T_eG^1$ is dense in $T_eG$.

\begin{lemma}[Geodesic equation]
Let $G$ be a right half-Lie group carrying a right-invariant Riemannian metric $g$. 
Then, existence of the Christoffel symbol $\Gamma$, defined in charts as
\begin{equation}
g(\Gamma(X,X),Y)
=
\frac12 dg(Y)(X,X)
-dg(X)(X,Y),
\qquad
X,Y\in T_xG,
\end{equation}
implies the existence of the transpose of the adjoint as a quadratic mapping  $TG^1\to TG$, defined as
\begin{equation}
g_e(\on{ad}_X^\top X,Y) 
=
g_e(X,\on{ad}_XY),
\qquad
X,Y\in T_eG^1.
\end{equation} 
\end{lemma}

\begin{proof}
For any $X,Y \in T_eG^1$, we compute in a chart
\begin{align}
g(\nabla_{R_X}R_X,R_Y)
&=
g(dR_X.R_X,R_Y)-g(\Gamma(R_X,R_X),R_Y)
\\&=
g(dR_X.R_X,R_Y)
-\tfrac12 dg(R_Y)(R_X,R_X)
+dg(R_X)(R_Y,R_X)
\\&=
g(dR_X.R_Y,R_X)
-g(dR_Y.R_X,R_X)
\\&=
-g([R_X,R_Y],R_X)
=
g(R_{[X,Y]},R_X)
=
g([X,Y],X).
\end{align}
Therefore, $\nabla_XR_X=\nabla_{R_X}R_X(e) \in T_eG$ satisfies the defining property of $\on{ad}_X^\top X$.
\end{proof}

Next, we show that the existence of solutions to the geodesic equations is passed on from $G$ to $G^k$, endowed with the induced (weak) Riemannian metric. 
Thus, there is no loss or gain of regularity along geodesics.
Similar results have been shown in many specific cases \cite{ebin1970groups, bruveris2017regularity, bauer2020fractional, bauer2022smooth}.
We will formulate the following theorem for general right invariant flows; the result for the geodesic equation follows by interpreting this equation as a flow equation (with respect to the geodesic spray) and noting that the geodesic spray is a right-invariant vector field.

\begin{theorem}[No-loss-no-gain]\label{thm:nolossnogain0} Let $G$ be a half-Lie group.
Let $S$ be a smooth vector field on $TG$ which is invariant for the right action of $G$ on $TG$ and which is a second order differential equation, i.e., $T(\pi_G)\o S = \on{Id}_{TG}$. Let $t_0\in (0,\infty]$ and let $U\subset  TG$ be a maximal open set such that 
the flow of $S$ exists as a map 
\begin{equation}
\operatorname{Fl}^S: (-t_0,t_0)\times U\to TG.
\end{equation}
Then the flow restricts to a smooth map
\begin{equation}
\operatorname{Fl}^S: (-t_0,t_0)\times \left(U\cap TG^k\right)\to TG^k
\end{equation}
for any $k\geq 1$, i.e., there is no gain or loss in regularity during the evolution along $S$.
\end{theorem}

The vector field  $S:TG\to T^2G$ is a \emph{spray} if it also satisfies $T(m_t^{G}).m^{TG}_t S(X) = S(tX)$ for the scalar multiplications $m^G_t$ on $TG\to G$ and $m^{TG}_t$ on the bundle $\pi_{TG}:T^2G\to TG$. Then $\pi_G\o\on{Fl}^S_t$ is a geodesic structure; see \cite[Section 22.7]{Michor08}, for example.

\begin{proof} Note that $TG$ is not a half-Lie group since it is not a group; it just carries a continuous right action of $G$.
Since $S$ is invariant  also the set $U$ is invariant under the right action of $G$.
For $X\in U$ and $0\le t<t_0$ we  have 
\begin{align}
\mu_{(\pi_G\o\on{Fl}^S_t)(X)}(y) &= (\mu^y\o \pi_G\o\on{Fl}^S_t)(X) = (\pi_G\o T\mu^y\o \on{Fl}^S_t)(X) 
\\&
= (\pi_G\o\on{Fl}^S_t\o T\mu^y)(X) = ((\pi_G\o\on{Fl}^S_t)\o R_X)(y).
\end{align}
If $X\in TG^k$ then $R_X:G\to TG$ is $C^k$. Since $\pi_G\o\on{Fl}^S_t: U\to G$ is $C^\infty$, the mapping 
$\mu_{(\pi_G\o\on{Fl}^S_t)(X)} = (\pi_G\o\on{Fl}^S_t)\o R_X: G\to G$ is $C^k$.   
Since $T(\pi_G)\o S = \on{Id}_TG$ we have 
\begin{align}
\p_t (\pi_G\o\on{Fl}^S_t)(X) &= T(\pi_G)(\p_t \on{Fl}^S_t(X)) = (T(\pi_G)\o S\o \on{Fl}^S_t)(X) =  \on{Fl}^S_t(X)
\end{align}
This implies that $\on{Fl}^S_t:U\cap TG^k\to TG$ is $C^k$ with inverse $\on{Fl}^S_{-t}$. Moreover,
\begin{align}
(\pi_G\o\on{Fl}^S_{s+t})(X) &= (\pi_G\o\on{Fl}^S_s\o \on{Fl}^S_t)(X) = (\pi_G\o\on{Fl}^S_s)(\p_t (\pi_G\o\on{Fl}^S_t)(X)),
\end{align} 
so $\mu_{(\pi_G\o\on{Fl}^S_t)(X)}$ has $\mu_{(\pi_G\o\on{Fl}^S_{-t}(X)}$ as $C^k$ inverse and thus $(\pi_G\o\on{Fl}^S_t)(X) \in G^k$ and thus $\on{Fl}^S_t:U\cap TG^k\to TG^k$.
\end{proof}

As a corollary of the above theorem we immediately obtain the following result concerning the exponential map of right invariant Riemannian metrics:
\begin{corollary}[No-loss-no-gain for exp]\label{thm:nolossnogain_exp} Let $G$ be a half-Lie group and let $g$ be a right invariant smooth Riemannian metric on $G$. If the Riemannian exponential mapping $\exp^g$ exists and is smooth on a maximal  open set $U=-U$ in $TG$,
then $\exp^g$ restricts to a smooth map
\begin{equation}
\exp^g:  \left(U\cap TG^k\right)\to TG^k
\end{equation}
for any $k\geq 1$, i.e., there is no gain or loss in regularity during the evolution along $S$.
\end{corollary}

\subsection*{The theorem of Hopf-Rinow.}
The following theorem is the second main result of our article; a Hopf--Rinow theorem for half-Lie groups. 
Recall that in finite dimensions, the Hopf--Rinow theorem asserts the equivalence of geodesic completeness, metric completeness, and geodesic convexity. 
For strong Riemannian manifolds in infinite dimensions, metric completeness implies geodesic completeness \cite{lang1999fundamentals, gay2015geometry}, but all other implications may fail \cite{grossman1965hilbert, mcalpin1965infinite, atkin1975hopf}. 
Not so for half-Lie groups with strong right-invariant metrics:
These are metrically and geodesically complete and, under a weak closure condition, also geodesically convex, as shown next.

\begin{theorem}[Hopf--Rinow on half-Lie groups]
\label{thm:hopf-rinow}
Let $G$ be a connected half-Lie group equipped with a right invariant strong Riemannian metric $g$, and let $d:G\x G\to \mathbb R_+$ be the induced geodesic distance on $G$.

Then the following completeness properties hold for $(G,g)$:
\begin{enumerate}[(a)]
\item\label{metriccomplete} the space $(G,d)$ is a complete metric
       space, i.e., every $d$-Cauchy sequences converge in $G$;
\item\label{geodesiccompleteI} the exponential map $\exp^g_e:T_eG \to G$ is defined on all of $T_eG$;
\item\label{geodesiccompleteII} the exponential map $\exp^g:TG\to G$ is defined on all of $TG$; 
\item\label{geodesiccompleteIII} the space $(G,g)$ is geodesically  complete, i.e., every geodesic is maximally definable on all of $\mathbb R$.
\end{enumerate}
Assume in addition that $G$ is $L^2$-regular and that for each $x\in G$ the sets
\begin{align}
\mathcal A_{x}:= \left\{\xi\in L^{2}([0,1],T_eG): \operatorname{evol}(\xi)=x\right\} \subset  L^{2}([0,1],T_eG)
\end{align}
are weakly closed. Then
\begin{enumerate}[(a),resume]
\item\label{geodesicconvex} the space $(G,g)$ is geodesically convex, i.e., any two points in $G$ can be connected by a geodesic of minimal length.
\end{enumerate}
In addition, the geodesic completeness statements, items\ref{geodesiccompleteI}--~\ref{geodesiccompleteIII},  hold  for the weak Riemannian manifolds $(G^k,g)$, where $g$ is the restriction of the Riemannian metric on $G$.
\end{theorem}

\begin{remark}
The closure condition for the sets $\mathcal A_x$ follows if $\operatorname{evol}: L^2([0,1],T_eG)\to G$ is continuous with respect to the weak topology on $L^2([0,1],T_eG)$ and some Hausdorff topology on $G$.
\end{remark}

\begin{proof}
By \cref{thm:riemannian:strong} we know that $G$ is a Hilbert manifold, and we will use this repeatedly throughout the proof.
To show~\ref{metriccomplete}, let $d$ be the geodesic distance in $G$, defined as the infinimum of the length of all piecewise smooth curves connecting two points. We know that $d$ is a metric whose induced topology is the manifold topology.
Let $\exp^g: TG\to G$ be the Riemannian exponential mapping.
Then
\[
\mathfrak g:= T_eG\supset B_\ep \xrightarrow[]{\exp^g_e} U_\ep\subset G
\]
is a diffeomorphism from a $g_e$-ball $B_\ep= \{X\in \mathfrak g: \|X\|_{g_e}<\ep\}$  in $\mathfrak g$ to an open neighborhood $U_\ep=\exp^g_e(B_\ep)$ of $e$ in $G$, for some $\ep>0$.
By the Gauss lemma \cite[Lemma 23.2, Corollary 23.3]{Michor08} we have
\begin{align}
U_\et:= \exp^g_e(B_\et)&=\{x\in G: d(e,x)<\et\} \text{ and }
\\
\bar U_\et:= \exp^g_e(\overline {B_\et})&=\{x\in G: d(e,x)\le\et\} \text{ for all  }0<\et<\ep.
\end{align}
We consider now the strong Riemannian metric
$\tilde g := (\exp^g_e)^*(g|_U)$ on $B_\ep\subset \mathfrak g$
and the smooth mapping
\[
B_\ep \ni X \mapsto  ((g_e^\vee)^{-1}\circ \tilde g_X^\vee : \mathfrak g = T_XB_\ep \to T_X^*B_\ep \to T_XB_\ep =\mathfrak g) \in L(\mathfrak g,\mathfrak g).
\]
For a suitably small ball $B_\eta$ we have
\[
\| (g_e^\vee)^{-1}\circ \tilde g_X\|_{L(\mathfrak g,\mathfrak g)} \le 2 \text{ and  }
\|( \tilde g_X^\vee )^{-1}\circ g_e^\vee\|_{L(\mathfrak g,\mathfrak g)} < 2
\]
for all $X\in B_\eta$. Thus the strong Riemann metrics $\tilde g$ and $g$ are uniformly bounded with respect to each other on $B_\eta$.
Next we consider a $d$-Cauchy-sequence $(x_i)$ in $G$.  For suitable $N$ we have $d(x_N,x_i)<\et/2$ for all $i\ge N$.
Then $(y_i:=x_i.x_N^{-1})_{i\ge N}$ is a Cauchy sequence in $U_{\et/2}$ since right translations are isometric, and
$(Y_i := (\exp^g_e)^{-1}(y_i)$ is a Cauchy sequence in $(B_{\et/2},\tilde g)$ and thus also in $(B_{\et/2}, g_e)$. Therefore, the sequence $(Y_i)$ has a limit $Y$ in $\bar B_\et/2$ in the $g_e$-norm, and
$y:=\exp^g_e(Y)$ is the limit of the sequence $(y_i)$, which concludes the proof of~\ref{metriccomplete}.

\ref{geodesiccompleteI}--\ref{geodesiccompleteIII} follow directly from metric completeness, as this assertation of the theorem of Hopf-Rinow is  still valid in infinite dimensions, see e.g.~\cite{lang1999fundamentals}. Alternatively, it can also be seen directly, as argued in \cite{gay2015geometry}:
As $g$ is strong, the geodesic exponential map is defined on some neighborhood of $0 \in T_eG$.
Moreover, as $g$ is right invariant, the right translation of a geodesic is a geodesic.
This allows one to extend a geodesic defined on an interval $(-s,s)$ to a geodesic defined on a larger interval $(-t,t)$ for some $t>s$.
It follows that the geodesic exponential map is defined on all of $T_eG$.

It remains to show that the space $(G,g)$ is geodesically convex. We will first show that there exist minimizing paths connecting any two elements $x_0$ and $x_1$.
The geodesic distance between $x_0$ and $x_1$ can be calculated by minimizing the Riemannian energy
\begin{equation}
E(x)=\int_0^1 g_x(x_t,x_t) dt,
\end{equation}
over all paths $x\in H^1([0,1],G)$, such that $x(0)=x_0$ and $x(1)=x_1$.
Using the right invariance of the Riemannian metric we can rewrite this as
\begin{equation}
E(x)=\int_0^1 \langle T\mu^{x^{-1}}.x_t,T\mu^{x^{-1}}.x_t\rangle dt
:=
\|T\mu^{x^{-1}}.x_t\|^2_{L^{2}([0,1],\mathfrak g)}.
\end{equation}
Thus, by the $L^2$-regularity of the half-Lie group $G$, i.e., the bijection between $H^1$-paths in the Lie group and $L^2$-paths in $\mathfrak g$,  we can reformulate the  calculation of the geodesic distance as the minimization problem
\begin{equation}
\inf\|\xi\|^2_{L^{2}([0,1],\mathfrak g)}
\end{equation}
where the infimum is taken over the set
\begin{multline}
\notag
 \mathcal B=\mathcal A_{x_1.x_0^{-1}}\cap\left\{
\xi\in L^{2}([0,1],\mathfrak g):
\|\xi\|^2_{L^{2}([0,1],\mathfrak g)}\leq 2 \operatorname{dist}(x_0,x_1)^2 \right\}= \\\left\{\xi\in L^{2}([0,1],\mathfrak g): \operatorname{evol}(\xi)=x_1.x_0^{-1}\right\}
\cap\left\{
\xi\in L^{2}([0,1],\mathfrak g):
\|\xi\|^2_{L^{2}([0,1],\mathfrak g)}\leq 2 \operatorname{dist}(x_0,x_1)^2 \right\}.
\end{multline}
By definition $\mathcal B$ is bounded.
Using the assumption that  $\mathcal A_{x_1.x_0^{-1}}$ is weakly closed it follows that $\mathcal B$ is also weakly closed and thus compact.

Next, we choose a sequence $\xi_n\in \mathcal B$ such that
\begin{equation}
\left|\inf_{\xi\in \mathcal B}\|\xi\|^2_{L^{2}([0,1],\mathfrak g)}-\|\xi_n\|^2_{L^{2}([0,1],\mathfrak g)}\right|<\frac{1}{n}
\end{equation}
As $B$ is compact, there exists  $\tilde \xi\in\mathcal B$ which is a cluster point for
$\xi_n$. By choosing a subsequence we may assume that $\xi_n\to \xi$ weakly.
Since every norm on a Banach space is sequentially weakly lower semi-continuous this implies that
\begin{equation}
\|\tilde \xi\|^2_{L^{2}([0,1],\mathfrak g)}=\inf_{\xi\in B}\|\xi\|^2_{L^{2}([0,1],\mathfrak g)}.
\end{equation}
Thus $x=\operatorname{Evol}(\xi)$ is an energy minimizing path. It remains to show that $x$ is a solution of the geodesic equation, but this follows by standard arguments: since $x$ is a minimizing path for $t\in [0,1]$, it is also minimizing on each subinterval and thus the statement follows by the Gauss lemma.

Finally we note, that the geodesic completeness for $G^k$ follows directly from the no-loss-no-gain result; Corollary~\ref{thm:nolossnogain_exp}.
\end{proof}

\begin{example}[Sobolev metrics on groups of diffeomorphisms]
For a compact, finite dimensional, Riemannian manifold $(M,g)$ we consider the group of Sobolev diffeomorphisms $\Diff_{H^s}(M)$ as introduced in Example~\ref{ex:diffeo}. We equip this infinite dimensional half-Lie group with the strong, right invariant Sobolev metric of order $s$, i.e.,
\begin{equation}
G^s_{\varphi}(h\circ\varphi,k\circ \varphi)=\int_M g((1-\Delta)^{s/2}h,(1-\Delta)^{s/2}k) \on{vol},
\end{equation}
where $h$ and $k$ are $H^s$ vector fields and where $\Delta$ ($\on{vol}$, resp.) are the Laplacian (volume form, resp.) of the (finite dimensional) Riemannian metric $g$. 
To see that $G$ is indeed a Riemannian metric, one only has to check that it depends smoothly on the foot point: for integer orders this is relatively easy as one can derive an explicit formula for the dependence on $\varphi$, see eg.~\cite{ebin1970groups}. For real $s$, on the other hand, this involves highly, non-trivial estimates and has been shown only recently~\cite{bauer2015local,bauer2020wellposedness}. Thus one immediately obtains the geodesic and metric completeness of $(\Diff_{H^s}(M),G^s)$. To prove the existence of minimizers one has to show the additional assumption in Theorem~\ref{thm:hopf-rinow}, which has been shown implicitly in~\cite{bruveris2017completeness}. Thereby we have recovered all completeness statements on the group of Sobolev diffeomorphisms as obtained in~\cite{bruveris2017completeness}. 
\end{example}

\subsection{Curvature}\label{curvature}
On a Lie group equipped with a right invariant Riemannian metric there are well developed formulas for curvature which for diffeomorphism groups are due to \cite{arnold1966}; but these make use of the Lie bracket on the Lie algebra, which we do not have in the case of half-Lie groups. Here we want to sketch another formula, due to  \cite{micheli2013sobolev}, which involves the inverse metric and locally constant 1-forms and make use of the \emph{symmetriced force} (from Newton's law $F=m.c''$ which lives in $T^2G$ and its dual) and \emph{stress} (which is half the Lie bracket). This formula is well adapted to the O'Neill formula for Riemannian submersions, and has been used to compute the curvature on landmark space in \cite{micheli2012sectional} and  \cite[Section 9.6]{micheli2013sobolev}. 

So let $G$ be a half-Lie group with a right invariant smooth Riemannian metric $g$. We consider smooth 1-forms $\al,\be:G\to g(TG)\subset T^*G$ in the case of a weak metric $g$. Note that for a strong metric we have $g(TG)=T^*G$. 
Then we  introduce auxiliary vector fields $X_\al$ and 
$X_\be$ playing the role of `locally constant' extensions of the value of $\al^\sharp= g\i\o \al$ and 
$\be^\sharp$ for the 1-forms $\al,\be$ at the point $x \in G$ where the curvature is being calculated and for which 
the 1-forms $\al, \be$ appear locally constant too. More precisely, assume we are given $X_\al$ and $X_\be$ such that:
\begin{enumerate}[(a)]
\item $X_\al(x) = \al^\sharp(x),\quad X_\be(x) = \be^\sharp(x)$,
\item Then $\al^\sharp-X_\al$ is zero at $x$ hence has a well defined derivative 
       $D_x(\al^\sharp-X_\al)$ lying in Hom$(T_xG,T_xG)$. For a vector field $Y$ we have 
       $D_x(\al^\sharp-X_\al).Y_x = [Y,\al^\sharp-X_\al](x) = \L_Y(\al^\sharp-X_\al)|_x$.
       The same holds for $\be$.
\item $\L_{X_\al}(\al)=\L_{X_\al}(\be)=\L_{X_\be}(\al)=\L_{X_\be}(\be)=0$,
\item $[X_\al, X_\be] = 0$.
\end{enumerate}
Locally constant 1-forms and vector fields with respect to a chart satisfy these properties. 
Using these forms and vector fields, we then define:
\begin{align}
\mathcal F(\al,\be) :&= \tfrac12 d(g\i(\al,\be)), \qquad \text{a 1-form on $M$ called the {symmetriced force},}\\
\mathcal D(\al,\be)(x) :&= D_x(\be^\sharp - X_\be).\al^\sharp(x)
\\&
= d(\be^\sharp - X_\be).\al^\sharp(x), \quad\text{a tangent vector at $x\in G$ called the {\it stress}.}
\end{align}  
Then in the notation above, by \cite[Section 2.3]{micheli2013sobolev}, we have the following formula for the the sectional curvature, where $\langle \al,X\rangle$ denote the evaluation of a 1-form on a tangent vector. 
\begin{align}
&g\big(R(\al^{\sharp},\be^{\sharp})\be^{\sharp},\al^{\sharp}\big)(x) = R_{11} +R_{12} + R_2 + R_3 \\ 
&\quad R_{11} = \tfrac12 \left(
\L_{X_\al}^2(g\i)(\be,\be)-2\L_{X_\al}\L_{X_\be}(g\i)(\al,\be)
+\L_{X_\be}^2(g\i)(\al,\al)
 \right)(x) \\
& \quad R_{12} = \langle \mathcal F(\al,\al), \mathcal D(\be,\be) \rangle + \langle \mathcal F(\be,\be),\mathcal D(\al,\al)\rangle  
- \langle \mathcal F(\al,\be), \mathcal D(\al,\be)+\mathcal D(\be,\al) \rangle \\
&\quad R_2 = \left(
\|\mathcal F(\al,\be)\|^2_{g\i}
-\big\langle \mathcal F(\al,\al)),\mathcal F(\be,\be)\big\rangle_{g\i} \right)(x) \\
&\quad R_3 = -\tfrac34 \| \mathcal D(\al,\be)-\mathcal D(\be,\al) \|^2_{g_x} \label{eq:curv3}
\end{align}

\section{Auxiliary results on local additions}
\label{sec:local_additions}

Local additions are defined in \cref{def:addition}. Here, we collect some further properties and explain their relation to sprays, linear connections, and geodesic structures. 

\begin{remark}[Local right-invariance of local additions]
The restriction of a right-invariant local addition $\ta$ to an open $e$-neighborhood $U\subseteq G$ is a smooth mapping $\ta^U:=\ta|_{TU\cap V}:TU\cap V\to G$ with the following properties: 
\begin{enumerate}[(a)]
\item $\ta^U(0_x)=x$ for all $x\in U$; 
\item $\ta^U(T_x\mu^y(X_x))= \mu^y(\ta^U(X_x))$ for $x,y\in G$ with $x\in U\cap U.y\i$ and $X_x\in T_xG\cap V$;
\item $(\pi_U,\ta^U):TU\cap V\to G\x G$ is a diffeomorphism onto its image.
\end{enumerate}
We call $\ta^U$ a \emph{local right invariant local addition near $e$.}
Obviously, we can reconstruct $\ta$ from $\ta^U$. 
\end{remark}

\begin{remark}[Normalization of local additions]
Every local addition $\ta:TG\supseteq V \to G$ can be normalized such that $\p_t|_0\tau(tX_x)=X_x$, for all $X_x \in TG$. To this aim, one has to replace $\tau$ by $\tau\circ A\i$, where $A$ is the fiber-wise linear map given by
\begin{align}
&A_x:T_xG\to T_xG, && A_x(X_x) = \p_t|_0 \ta(tX_x) =  (T_{0_x}\ta \o \on{vl})(0_x,X_x).
\end{align}
Here, $\on{vl}:TG\x_G TG\to V(TG)\subset T^2G$ is the vertical lift, which is given by $\on{vl}(X_x,Y_x) = \p_t|_0 X_x +t Y_x$; see \cite[Section 29.9]{KrieglMichor97} or \cite[Section 8.12]{Michor08}. The map $A\in L(TG,TG)$ is fiber-wise invertible with inverse 
$A_x\i = \on{vpr}\o T_{0_x}(\ta|_{T_xG\cap V})\i$ because $(\pi_G,\ta):V\to G\x G$ is a diffeomorphism onto its image. 
\end{remark}

\begin{remark}[Sprays, linear connections, and geodesic structures]
Local additions are closely related to sprays, linear connections, and geodesic structures; see \cite[Sections 22.6--22.8]{Michor08} for definitions and notations. 
Let $\ta:TG\supseteq V\to G$ be a local addition. 
Without loss of generality, $\ta$ is already normalized.   
Then, the vector field
\begin{align}
S:TG&\to T^2G, \qquad S(X_x): = \p_t^2|_0 \ta(tX_x),
\end{align}
satisfies the defining properties of a \emph{spray}:
\begin{align}
\pi_{TG} S(X_x) &= \pi_{TG}\p_t^2|_0\ta(tX_x) = \p_t|_0 \ta(tX_x) = X_x,
\\
T(\pi_{G}) S(X_x) &= T(\pi_{G})\p_t^2|_0\ta(tX_x) = \p_t|_0 (\pi_G \p_t \ta(t X_x)) = \p_t|_0 (\ta(t X_x)) = X_x,
\\
S(sX_x) &= \p_t^2|_0 \ta(tsX_x) = \p_t|_0 \big(m^{TG}_s\p_r|_0 \ta(trsX_x)\big) = m^{T^2G}_s T(m^{TG}_s) S(X_x),
\end{align}
where $m^{TG}_s$ is scalar multiplication by $s\in \mathbb R$ on the bundle $TG\to G$.
The spray $S$ can equivalently be expressed as a symmetric \emph{linear connection}.  
If $G$ is a Banach manifold, then the flow of $S$ exists locally and gives rise to a \emph{geodesic structure} 
$\on{geo}(X)(t) = (\pi_G\o \on{Fl}^S_t)(X)$. 
A geodesic structure is a local addition with some extra properties. 
If the local addition $\ta$ is right-invariant, then the spray, linear connection, and geodesic structure are also right-invariant.
\end{remark}

\section{Auxiliary results on jets}
\label{sec:jets}

This section collects several technical results on jets and differentiability of functions for subsequent use.
See \cite[Section~41]{KrieglMichor97} for unexplained standard notation.
Throughout \cref{sec:jets}, $x,y,z$ are points in open subsets $U,V,W$ of Banach spaces $E,F,G$, respectively, $f:U\to V$ and $g:V\to W$ are $k$ times differentiable functions, and $k \in \mathbb N$.

\begin{definition}[Jets]
\label{def:jets}
The space of $k$-jets from $U$ to $V$ is defined by
\[
J^k(U,V) := U \x V \x \on{Poly}^k(E,F),
\quad
\text{where}
\quad
\on{Poly}^k(E,F)
=
\prod_{j=1}^k L^j_{\on{sym}}(E;F)
\]
carries the sum of the norms for multilinear mappings.
It is an open subset of the Banach space $J^k(E,F)$.
We write $\al$ and $\be$ for the source and target projections, i.e.,
\begin{equation}
\al=\on{pr}_1:J^k(U,V)\to U,
\qquad
\be=\on{pr}_2:J^k(U,V)\to V.
\end{equation}
The $k$-jet of a function $f:U\to V$ at $x \in U$ is defined as
\[
j^k_xf
=
\Big(x,f(x),df(x),\frac{1}{2!}d^2f(x),\dots,\frac{1}{k!}d^kf(x)\Big)
\in
J^k(U,V).
\]
\end{definition}

\begin{definition}[Jet composition]
\label{def:jet_composition}
On the fibered product
\begin{equation}
J^k(V,W) \x_V J^k(U,V)
=
\big\{(\tau,\sigma) \in J^k(V,W) \x J^k(U,V):\al(\ta)=\be(\si)\big\},
\end{equation}
jet composition $\bullet: J^k(V,W) \x_V J^k(U,V) \to J^k(U,W)$ is defined as
\[
(y,z,q)\bullet(x,y,p)
=
\big(x,z,\pi_k(q\o p)\big),
\]
where $\pi_k$ discards all monomials of order zero or greater than $k$, and where $\o$ denotes composition of polynomials.
\end{definition}

\begin{lemma}[Jet composition]
\label{lem:jet_composition}
Jet composition satisfies the defining property
\[
j^k_{f(x)}g\bullet j^k_x f=j^k_x(f\o g)
\]
and satisfies the bounds
\begin{align}
\|\ta\bullet\si\|
&\leq
(1+\|\ta\|)(1+\|\si\|^k),
\\
\|\tilde\ta\bullet\tilde\si-\ta\bullet\si\|
&\leq
\|\tilde \ta-\ta\|(1+\|\tilde\si\|^k)
+(1+\|\ta\|) \|\tilde \si-\si\|(1+k\|\tilde\si\|^{k-1}+k\|\si\|^{k-1}).
\end{align}
\end{lemma}

\begin{proof}
The defining property is the multi-dimensional version (see \cite[Section 2.4]{MM13}, e.g.) of Fa\`a di Bruno's formula \cite{bruno1857note}, which states that the Taylor series of a composition of two functions is the composition of their respective Taylor series. We will also use that multi-composition of mulilinear mappings is a bounded operation for Banach spaces.
In the subsequent computations, $\sigma=(x,y,p)\in J^k(U,V)$, $\tau=(y,z,q)\in J^k(V,W)$, and all indices run from $1$ to $k$.
\begin{align}
\|q\o p\|
&=
\Big\|\sum_jq_j\Big(\sum_{i_1}p_{i_1},\dots,\sum_{i_j}p_{i_j}\Big)\Big\|
\leq
\sum_j\|q_j\|\Big(\sum_{i_1}\|p_{i_1}\|\Big)\cdots\Big(\sum_{i_j}\|p_{i_j}\|\Big)
\\&=
\sum_j\|q_j\|\|p\|^j
\leq
\|q\|\max_j\|p\|^j
\leq
\|q\|(1+\|p\|^k),
\end{align}
\begin{align}
\|\ta\bullet\si\|
&=
\|(x,z,\pi_k(q\o p))\|
\leq
\|x\|+\|z\|+\|q\|(1+\|p\|^k)
\leq
(1+\|\ta\|)(1+\|\si\|^k),
\end{align}
\begin{align}
\|q\o\tilde p-q\o p\|
&=
\Big\|\sum_j\sum_{\ell=1}^j q_j\Big(\sum_{i_1}p_{i_1},\dots,\sum_{i_\ell}\tilde p_{i_\ell}-p_{i_\ell},\dots\sum_{i_j}p_{i_j}\Big)\Big\|
\\&\leq
\sum_j\sum_{\ell=1}^j \|q_j\|\|\tilde p\|^{\ell-1}\|\tilde p-p\|\|p\|^{j-\ell}
\\&\leq
\|q\| \max_j \sum_{\ell=1}^j \|\tilde p\|^{\ell-1}\|\tilde p-p\|\|p\|^{j-\ell}
\\&\leq
\|q\| \|\tilde p-p\|(1+k\|\tilde p\|^{k-1}+k\|p\|^{k-1}),
\end{align}
\begin{align}
\|\tilde q\o \tilde p-q\o p\|
&\leq
\|(\tilde q-q)\o \tilde p\|+\|q\o\tilde p-q\o p\|
\\&\leq
\|\tilde q-q\|(1+\|\tilde p\|^k)
+\|q\| \|\tilde p-p\|(1+k\|\tilde p\|^{k-1}+k\|p\|^{k-1}),
\end{align}
\begin{align}
\|\tilde\ta\bullet\tilde\si-\ta\bullet\si\|
&=
\big\|\big(\tilde x-x,\tilde z-z,\pi_k(\tilde q\o\tilde p-q\o p)\big)\big\|
\\&\leq
\|\tilde z-z\|
+\|\tilde q-q\|(1+\|\tilde p\|^k)
\\&\qquad
+\|\tilde x-x\|
+\|q\| \|\tilde p-p\|(1+k\|\tilde p\|^{k-1}+k\|p\|^{k-1})
\\&\leq
\|\tilde \ta-\ta\|(1+\|\tilde\si\|^k)
+(1+\|\ta\|) \|\tilde \si-\si\|(1+k\|\tilde\si\|^{k-1}+k\|\si\|^{k-1}).
\qedhere
\end{align}
\end{proof}

\begin{definition}[Jet evaluation]
\label{def:jet_evaluation}
Let $\pi_H$ and $\pi_V$ denote the projections onto the first and second component of $TU\cong U\x E$.
On the fibered product
\begin{equation}
J^k(U,V) \x_U TU
=
\big\{(\sigma,\xi) \in J^k(U,V) \x TU:\al(\si)=\pi_H(\xi)\big\},
\end{equation}
jet evaluation $\odot: J^k(U,V) \x_U TU \to J^{k-1}(TU,TV)$ is defined as
\begin{multline}
\notag
(x,y,p)\odot(x,X)
=
\Big((x,X),\big(y,p_1(X)\big),\dots,
\\
\big(p_{k-1}(\pi_H,\dots,\pi_H),p_k(\pi_H,\dots,\pi_H,X)+(k-1) p_{k-1}(\pi_H,\dots,\pi_H,\pi_V)\big)\Big).
\end{multline}
\end{definition}

\begin{lemma}[Jet evaluation]
\label{lem:jet_evaluation}
Jet evaluation satisfies the defining property
\[
j^k_xf\odot\xi = j^{k-1}_\xi Tf,
\qquad
\xi \in T_xU
\]
and satisfies the bounds
\begin{align}
\|\si\odot\xi\|
&\leq
\|\xi\|+(k+1)\|\sigma\|+\|\xi\|\|\sigma\|,
\\
\|\tilde\si\odot\tilde \xi-\si\odot \xi\|
&\leq
\|\tilde\si-\si\|(k+\|\tilde \xi\|)+\|\tilde \xi-\xi\|(1+\|\sigma\|).
\end{align}
\end{lemma}

\begin{proof}
To verify the defining property, one starts from the expression
\begin{equation}
Tf(x,X)
=
\big(f(x),df(x)(X)\big)
\end{equation}
and computes iterated derivatives with respect to $(x,X)$. These are given by:
\begin{align}
&\big(df(x)(\pi_H),d^2f(x)(\pi_H,X)+df(x)(\pi_V)\big),
\\&
\big(d^2f(x)(\pi_H,\pi_H),d^3f(x)(\pi_H,\pi_H,X)+2d^2f(x)(\pi_H,\pi_V)\big),
\\&
\big(d^3f(x)(\pi_H,\pi_H,\pi_H),d^4f(x)(\pi_H,\pi_H,\pi_H,X)+3d^3f(x)(\pi_H,\pi_H,\pi_V)\big),
\\&\dots
\\&
\big(d^{k-1}f(x)(\pi_H,\dots,\pi_H),d^kf(x)(\pi_H,\dots,\pi_H,X)+(k-1)d^{k-1}f(x)(\pi_H,\dots,\pi_H,\pi_V)\big).
\end{align}
These terms coincide with the expression of $\sigma\odot\xi$ with $\sigma=(x,y,p)=j^k_xf$ and $\xi=(x,X)$.
To verify the first bound, one collects for any $\ell \in \{1,\dots,k-1\}$ the terms in the formula for $\|\sigma\odot\xi\|$ involving $p_\ell$. These terms are:
\begin{equation}
p_\ell(\pi_H,\dots,\pi_H,X),
\qquad
p_\ell(\pi_H,\dots,\pi_H),
\qquad
\ell p_\ell(\pi_H,\dots,\pi_H,\pi_V).
\end{equation}
The sum of the norms of these terms is bounded by $\|p_\ell\|(\|X\|+k)$. Thus, overall, one obtains
\begin{equation}
\|\sigma\odot\xi\|
\leq
\|\sigma\|+\|\xi\|+\|\sigma\|(\|\xi\|+k)
=
\|\xi\|+(k+1)\|\sigma\|+\|\xi\|\|\sigma\|.
\end{equation}
For the second bound, one estimates
\begin{align}
\hspace{2em}&\hspace{-2em}
\|\tilde\si\odot\tilde \xi-\si\odot \xi\|
\leq
\|\tilde x-x\|
+\|\tilde X-X\|
+\|\tilde y-y\|
+\|(\tilde p_1-p_1)(\tilde X)\|
+\|p_1(\tilde X-X)\|
\\&\qquad
+\dots
+\|(\tilde p_{k-1}-p_{k-1})(\pi_H,\dots,\pi_H)\|
+\|(\tilde p_k-p_k)(\pi_H,\dots,\pi_H,\tilde X)\|
\\&\qquad
+\|p_k(\pi_H,\dots,\pi_H,\tilde X-X)\|
+(k-1)\|(\tilde p_{k-1}-p_{k-1})(\pi_H,\dots,\pi_H,\pi_V)\|
\\&\leq
\|\tilde\si-\si\|(k+\|\tilde \xi\|)+\|\tilde \xi-\xi\|(1+\|\sigma\|).
\qedhere
\end{align}
\end{proof}

\begin{definition}[Jet inversion]
\label{def:jet_inversion}
On the set of invertible $k$-jets,
\begin{equation}
J^k(U,V)^\times = U\times V\times GL(E,F)\times \prod_{j=2}^k L^j_{\on{sym}}(E;F),
\end{equation}
inversion $(\cdot)^{-1}:J^k(U,V)^\times\to J^k(V,U)^\times$ is defined as
\begin{equation}
(x,y,p_1,\dots,p_k)^{-1}=\big(y,x,q_1,\dots,q_k\big),
\end{equation}
where $(q_1,\dots,q_k)$ is defined recursively as
\begin{equation}
q_1 = p_1^{-1} \in GL(F,E),
\qquad
q_k = - \sum_{j=1}^{k-1}\sum_{\substack{1\leq i_1,\dots,i_j\leq k\\i_1+\dots+i_j=k}} q_j(p_{i_1}\o q_1,\dots,p_{i_j}\o q_1).
\end{equation}
\end{definition}

\begin{lemma}[Jet inversion]
\label{lem:jet_inversion}
Jet inversion is continuous and satisfies for any diffeomorphism $f$ the defining property
\begin{equation}
(j^k_xf)^{-1} = j^k_{f(x)}(f^{-1}).
\end{equation}
\end{lemma}

\begin{proof}
The continuity of jet inversion follows from the continuity of the inversion $GL(E,F)\to GL(F,E)$ and the continuity of the composition map in the category of multi-linear functions between Banach spaces.
By $k$-fold differentiation of the equation $g\o f=\Id_U$, one obtains for $k=1$ that $(g^{(1)}\o f)(f^{(1)})=\Id_E$ and for $k\geq 2$ that
\begin{equation}
\sum_{j=1}^k\sum_{\substack{1\leq i_1,\dots,i_j\leq k\\i_1+\dots+i_j=k}} (g^{(j)}\o f)(f^{(i_1)},\dots,f^{(i_j)})=0.
\end{equation}
Setting $p_j=f^{(j)}$ and $q_j=g^{(j)}\o f$ and solving for $q_k$ yields the defining property.
\end{proof}

\begin{lemma}
\label{lem:curves}
The space of smooth curves $c\in C^\infty(\mathbb R,C^k(U,V))$ is isomorphic to the space of all $\hat c: \mathbb R\x U\to V$ such that for all $i\in \mathbb N$ and $j\in\mathbb N_{\leq k}$, the iterated derivative 
$\p_t^i d_U^j\hat c: \mathbb R\x U\to L^j(E,F)$ exists and is continuous. 
\end{lemma}

\begin{proof}
$\lim_{s\to 0}\frac1s (c(t+s)-c(t)) = c'(t)$ exists and converges in $C^k(U,F)$ if and only 
\[
\lim_{s\to 0} \left\| \tfrac1{s}\big(d^j_V\hat c(t+s,x)-d^j_V\hat c(t,x)\big) - \p_t d^j_V\hat c(t,x)\right\|_{L^j_{\text{sym}}(E;F)} = 0\quad
\text{ for  }0\le j\le k
\]
uniformly for $(t,x)$ in compact subsets of $\mathbb R\x U$, since open subsets of Banach spaces are compactly generated.
By iteration this holds also for $\p_t^i$ instead of $\p_t$.
\end{proof}

\begin{lemma}[Pull-backs and push-forwards]
\label{lem:Ck_calculus}
\begin{enumerate}[(a)]
\item\label{lem:Ck_calculus:1} The pull-back $f^*:C^k(V,W)\ni g \mapsto g\o f \in C^k(U,W)$ along any function $f\in C^k(U,V)$ is smooth.
\item\label{lem:Ck_calculus:2} The push-forward $g_*:C^k(U,V)\ni f \mapsto g\o f \in C^k(U,W)$ along any function $g \in C^\infty(V,W)$ is smooth.
\item\label{lem:Ck_calculus:3} The push-forward $g_*:C^k(U,V)\ni f \mapsto g\o f \in C^k(U,W)$ along any function $g\in C^{k+\ell}(V,W)$ is 
$C^\ell$.
\item\label{lem:Ck_calculus:4} The composition $C^k(U,V)\times C^k(V,W)\ni (f,g)\mapsto g\o f \in C^k(U,W)$ is sequentially continuous.
\end{enumerate}
\end{lemma}

\begin{proof}
As usual, the topology on $C^k(V,W)$ is induced via $k$-jets from the compact-open topology on $C(V,W\x \on{Poly}^k(F,G))$.
Here, $\on{Poly}^k(F,G) =\prod_{j=1}^k L^j_{\on{sym}}(F;G)$ is the Banach space described in \cref{def:jets}. 
\begin{enumerate}[(a),wide]
\item For $f\in C^k(U, V)$ the pull-back mapping 
\begin{alignat}{5}
f^* &= C^k(g, W) &&: C^k(V, W)&&\to C^k(U, W)\quad&&\text{is the restriction of the linear mapping }
\\ 
f^* &= C^k(g, G) &&: C^k(V, G)&&\to C^k(U, G),\quad&&\text{which is bounded, thus smooth. }
\end{alignat}

\item For $g\in C^{\infty}(V,W)$, the push-forward $g_* = C^k(U,g): C^k(U,V)\to C^k(U,W)$ maps smooth curves to smooth curves by \cref{lem:curves} and the iterated chain rule.
By the principles of convenient analysis \cite{KrieglMichor97}, the mapping $g_*$ is $C^\infty$ in the usual sense since the spaces $C^k(U,F)$ are inverse limits of Banach spaces and thus the $c^\infty$-topology coincides with the usual one. 

\item 
We first consider $\ell=0$.  
By Fa\`a di Bruno's formula, for any $m\leq k$,
\begin{equation}
\partial_x^m g(f(x))
=
\sum_{n=1}^m 
g^{(n)}(f(x))
B_{m,n}(f^{(1)}(x),\dots,f^{(m-n+1)}(x)),
\end{equation}
where $B_{m,n}$ are the Bell polynomials. 
The function
\begin{equation}
C(U,V) \ni f \mapsto j^{n}_fg \in C(U,\on{Poly}^n(F,G))
\end{equation}
is continuous in the compact-open topologies. 
Thus, overall, one has a continuous map 
\begin{align}
C^k(U,V) \ni f \mapsto g(f(x)) \in C^k(U,G).
\end{align}
This proves the statement for $\ell=0$.
We next consider $\ell=1$. To identify a candidate derivative of $g_*$, we compute for any $x \in U$
\begin{equation}
\label{equ:pushforward:evx}
d(\on{ev}_x\circ g_*)(f)(h)
=
\lim_{t\to 0} \frac1t\big(g(f(x)+th(x))-g(f(x))\big)
=
g'(f(x))h(x).
\end{equation}
Thus, if $dg_*$ is differentiable, then its derivative is given by
\begin{equation}
\label{equ:candidate_derivative}
\xymatrix@R=0pt{
C^k(U,V)
\ar[r]
&
C^k(U,L(F,G))
\ar[r]
&
L(C^k(U,F),C^k(U,G))
\\
*\txt<4.5em>{$f$}
\ar@{|->}[r]
&
*\txt<7em>{$g'\o f$}
\ar@{|->}[r]
&
*\txt<10.5em>{$(h\mapsto (g'\o f).h)$}
}
\end{equation}
The continuity of the left arrow follows from the already established case $\ell=0$.
The continuity of the right arrow can be seen as follows. 
For any compact set $K\subseteq U$, the function
\begin{equation}
C^k(K,L(F,G))\times C^k(K,F) 
= 
C^k(K,L(F,G)\times F)
\xrightarrow{\on{ev}_*}
C^k(K,G)
\end{equation}
is continuous because the evaluation map $\on{ev}$ is linear, and differentiation commutes with linear maps.
By the exponential law for continuous linear maps between Banach spaces, this is equivalent to the continuity of the function
\begin{equation}
C^k(K,L(F,G))\to L(C^k(K,F),C^k(K,G)).
\end{equation}
As the restriction $C^k(U,\cdots)\to C^k(K,\cdots)$ is continuous, one obtains continuity of the function
\begin{equation}
C^k(U,L(F,G))\to L(C^k(U,F),C^k(K,G)).
\end{equation}
As $K$ is arbitrary, and $C^k(U,G)$ carries the initial topology with respect to all restrictions to $C^k(K,G)$, one obtains continuity of the function
\begin{equation}
C^k(U,L(F,G))\to L(C^k(U,F),C^k(U,G)). 
\end{equation}
This establishes the continuity of the candidate derivative \eqref{equ:candidate_derivative} in the sense of Fr\'echet. 
Evaluating the candidate derivative along the curve $s\mapsto(f+sh,h)$ results in a continuous curve, whose Riemann integral is determined by \eqref{equ:pushforward:evx} and given by 
$$
\int_0^t (g'\circ(f+sh))(h) ds
= 
g\o(f+th)-g\o f.
$$
It follows that $g_*$ is Gateaux differentiable at $f$ in the direction $h$ with derivative $(g'\o f)(h)$.
Overall, we have shown that $g_*$ is continuously Fr\'echet differentiable.
This proves the case $\ell=1$.
Finally, for $\ell>1$, one proceeds by induction, using the fact that the Fr\'echet derivative of $g_*$ is a push-forward \eqref{equ:candidate_derivative}.
\item 
We first consider the case $k=0$; see \cite[Proposition 13.9 (c)]{Glockner20}.
Let $f_n\to f$ and $g_n\to g$ in the compact-open topologies, let $K$ be a compact subset of $U$, and let $O$ be an open subset of $V$ which contains $g(f(K))$. Then, the set $L=f(K)\cup\bigcup_n f_n(K)$ is compact.
Thus, for sufficiently large $n$, $g_n\o f_n(K)\subseteq g_n(L)\subseteq O$. 
As $K$ and $O$ were arbitrary, $g_n\o f_n$ converges in the compact-open topology to $g\o f$.
This proves the case $k=0$. 
Applying this result to the Fa\`a di Bruno formula as in \ref{lem:Ck_calculus:3} proves case $k>0$.
\qedhere
\end{enumerate} 
\end{proof}

\begin{definition}[Jets between manifolds]
Let $M$ and $N$ be manifolds with atlas $(U_i,u_i)$ and $(V_j,v_j)$, respectively.
Then $J^k(M,N)$ is the manifold obtained by gluing together the open sets $J^k(u_i(U_i),v_j(V_j))$ via the chart change mappings
\begin{multline}
\notag
J^k(u_{i'}(U_i\cap U_{i'}),v_{j'}(V_j\cap V_{j'}))
\ni \sigma \mapsto
\\
\mapsto j^k(v_j\o v_{j\i}) \bullet \sigma \bullet j^k(u_{i'}\o u_i\i) \in
J^k(u_i(U_i\cap U_{i'}),v_j(V_j\cap V_{j'})).
\end{multline}
\end{definition}

\begin{definition}[Jets of sections of fiber bundles]
Let $M$ be a manifold with atlas $(U_i,u_i)$,
and let $p:N\to M$ be a fiber bundle over $M$ with standard fiber $S$, fiber bundle atlas $(U_i,\ps_i:N|U_i\to U_i\times S)$, and chart changes $(\ps_{ij}=\on{pr}_2\o\ps_i\o\ps_j\i:U_i\cap U_j\x S\to S)$.
Then $J^k(E)$ is the manifold obtained by gluing together the manifolds $J^k(U_i,S)$ via the chart change mappings
\begin{equation}
J^k(U_i\cap U_j,S) \ni \sigma \mapsto j^k(\ps_{ij}(\al(\si),\cdot))\bullet \sigma \in J^k(U_i\cap U_j,S).
\end{equation}
The source projection $\al:J^k(N)\to M$ defines a fiber bundle over $M$.
\end{definition}

\section{Auxiliary results on right-invariant functions}
\label{sec:CkGGG}
The methods developed in this section are used to identify the half-Lie group structure on the groups $G^k$ of differentiable elements in $G$.
Throughout \cref{sec:CkGGG}, $G$ is a Banach right half-Lie group with identity element $e \in G$, and $k \in \mathbb N$.

Left-translation by any group element is a right-invariant function.
Conversely, every right-invariant function is a left-translation by some group element.
This motivates the following definition.

\begin{definition}[Right-invariant $C^k$ functions]
\label{def:CkGGG}
A function $f:G\to G$ is called right-invariant if $f\o\mu^y=\mu^y\o f$ for all $y \in G$.
For any $k \in \mathbb N$, let $\Diff_{C^k}(G)^G$ denote the space of right-invariant $C^k$ diffeomorphisms on $G$, endowed with the compact-open topology of $C(G,J^k(G,G))$.
\end{definition}

The main result of this section is \cref{thm:CkGGG}, which establishes that $\Diff_{C^k}(G)^G$ is a Banach half-Lie group if $G$ admits a right-invariant local addition.
This result would be wrong without the condition of right-invariance.
Indeed, if $G$ is infinite-dimensional, then $\Diff_{C^k}(G)$ is not a Banach manifold, and the composition $\Diff_{C^k}(G)\times \Diff_{C^k}(G)\to \Diff_{C^k}(G)$ is discontinuous.

\begin{theorem}[Right-invariant $C^k$ diffeomorphisms]
\label{thm:CkGGG}
Let $G$ be a Banach right half-Lie group carrying a right-invariant local addition.
Then, for any $k \in \mathbb N$, the space $\Diff_{C^k}(G)^G$ is a Banach half-Lie group with right-invariant local addition, and evaluation at the identity element $e \in G$ is smooth:
\begin{equation}
\on{ev}_e:\Diff_{C^k}(G)^G\ni f\mapsto f(e)\in G.
\end{equation}
\end{theorem}

\begin{proof}
This follows from \cref{lem:CkGGG_manifold,lem:CkGGG_topological_group} below.
\end{proof}

\begin{lemma}[Local boundedness of right multiplication]
\label{lem:multiplication_bounded}
There exists a neighborhood $U$ of $e \in G$, contained in some chart domain, such that for all $k \in \mathbb N$ and all compact subsets $K \subseteq G$,
\begin{align}
\sup_{x \in K}\sup_{y \in U} (\|j^k_x\mu^y\|+\|j^k_x\mu^{y\i}\|) < \infty,
\end{align}
with respect to the norm induced from $T_eG$ via the chart.
\end{lemma}

\begin{proof}
Let $U$ be the domain of some chart around $e \in G$.
Via this chart, $U$ can be identified with an open subset of the Banach space $E=T_eG$.
By shrinking $U$ if necessary, we achieve that $U$ is bounded in $E$.
Using the continuity of inversion and multiplication, we achieve additionally that $U=U\i$ and that $\mu(U,U)$ is bounded in $E$.
Then, for any $x \in U$, the set $\{\mu^y(x): y \in U\}$ is bounded in $E$.
As the bornology on $C^\infty(U,E)$ is generated by point evaluations, the set $\{\mu^y:y \in U\}$ is bounded in $C^\infty(U,E)$.
As jet prolongation is bounded linear, the set $\{j^k\mu^y:y \in U\}$ is bounded in $C^\infty(U,J^k(E,E))$.
In particular, the set $\{j^k_x\mu^y:x \in K,y\in U\}$ is bounded in $J^k(E,E)$.
This set coincides with $\{j^k_x\mu^{y\i}:x \in K,y\in U\}$ because $U=U\i$.
Boundedness in the normed space $J^k(E,E)$ can be expressed equivalently as in the statement of the lemma.
\end{proof}

\begin{definition}[Right-invariant $C^k$ vector fields]
\label{def:XCkG}
A vector field $X:G\to TG$ is called right-invariant if $X\o\mu^y=T_e\mu^y\o X$ for all $y \in G$.
For any $k \in \mathbb N$, let $\mathfrak X_{C^k}(G)^G$ denote the space of right-invariant $C^k$ vector fields on $G$, endowed with the compact-open topology of $C(G,J^k(TG))$.
\end{definition}

\begin{lemma}[Jets of right-invariant vector fields]
\label{lem:right-invariance}
For any $k \in \mathbb N$
and right-invariant $k$-times differentiable vector field $X$, the $k$-jet at $y \in G$ is uniquely determined by the $k$-jet at $e \in G$ as follows:
\begin{align}
j^k_yX
&=
(j^{k+1}_e\mu^y \odot X(e))\bullet j^k_eX \bullet j^k_y \mu^{y\i},
\end{align}
for some smooth binary operations $\bullet$ and $\odot$, which are defined in \cref{sec:jets}.
\end{lemma}

\begin{proof}
Thanks to the right-invariance of $X$, one has
\begin{align}
j^k_yX
&=
j^k_yX \bullet j^k_e \mu^y \bullet j^k_y \mu^{y\i}
=
j^k_e(X\o\mu^y)\bullet j^k_y \mu^{y\i}
\\&=
j^k_e(T\mu^y\o X)\bullet j^k_y \mu^{y\i}
=
j^k_{X(e)}T\mu^y\bullet j^k_eX \bullet j^k_y \mu^{y\i}
\\&=
(j^{k+1}_e\mu^y \odot X(e))\bullet j^k_eX \bullet j^k_y \mu^{y\i}.
\qedhere
\end{align}
\end{proof}

\begin{corollary}[Bound on jets of right-invariant vector fields]
\label{lem:bound}
Let $k \in \mathbb N$, and let $X$ and $Y$ be $k$-times differentiable right-invariant vector  fields on $G$.
Then, for any $y$ in the domain of some chart around $e \in G$, one has
\begin{align}
\|j^k_yX-j^k_yY\|
\leq
\|j^k_eX-j^k_eY\| p(\|j^k_eX\|,\|j^k_eY\|,\|j^{k+1}_e\mu^y\|,\|j^k_y\mu^{y\i}\|),
\end{align}
for some polynomial $p$ depending only on $k$, in the norm induced from $T_eG$ via the chart.
\end{corollary}

\begin{proof}
This follows from the estimates of \cref{lem:jet_composition,lem:jet_evaluation} applied to the expression in \cref{lem:right-invariance}.
\end{proof}

\begin{lemma}[Right-invariant $C^k$ vector fields]
\label{lem:banach}
For any $k \in \mathbb N$, the space $\mathfrak X_{C^k}(G)^G$ is Banach with respect to the norm
\[
\|X\| := \|X(e)\|_{T_eG} + \dots + \|d^kX(e)\|_{L^{(k)}(T_eG,\dots,T_eG;T_eG)}.
\]
\end{lemma}

\begin{proof}
The compact-open topology on $k$-jets dominates the norm topology.
To show that it coincides with the norm topology, let $(X_n)$ be a sequence which converges in norm to $X \in \mathfrak X_{C^k}(G)^G$.
By \cref{lem:multiplication_bounded}, there is a neighborhood $U$ of $e \in G$ where right translations are uniformly bounded on compacts in the sense detailed there.
It follows from \cref{lem:bound} that $j^kX_n$ converges to $j^kX$ uniformly on compacts in $U$.
This is equivalent to convergence in the compact-open topology on $C(U,J^k(G,G))$.
By using right translations, similarly to \cref{lem:right-invariance}, one obtains convergence in the compact-open topology on $C^k(\mu^y(U),J^k(G,G))$, for any $y \in G$.
As the sets $\mu^y(U)$ form an open cover of $G$, one obtains convergence in the compact-open topology on $C^k(G,J^k(G,G))$.
By definition, this means convergence in $\mathfrak X_{C^k}(G)$.
A similar argument shows that the norm topology is complete: for any Cauchy sequence $(X_n)$, it follows from \cref{lem:multiplication_bounded,lem:bound} that $\|j^kX_n-j^kX_m\|\to 0$ uniformly on compacts in $U$, which implies the existence of a limit $X \in \mathfrak X_{C^k}(G)$.
\end{proof}

\begin{lemma}[Banach manifold of right-invariant $C^k$ diffeomorphisms]
\label{lem:CkGGG_manifold}
Let $G$ carry a right-invariant local addition.
Then, for any $k \in \mathbb N$, the space $\Diff_{C^k}(G)^G$ is a Banach manifold with right-invariant local addition.
For any $g\in \Diff_{C^k}(G)^G$, the pull-back $g_*:\Diff_{C^k}(G)^G\ni f \mapsto f\o g \in \Diff_{C^k}(G)^G$ is smooth.
Moreover, the evaluation map $\on{ev}_e:\Diff_{C^k}(G)^G\to G$ is smooth.
\end{lemma}

\begin{proof}
A right-invariant $C^k$ function $f:G\to G$ belongs to $\Diff_{C^k}(G)^G$ if and only if $T_ef$ is continuously invertible.
Indeed, by right-invariance, this implies invertibility of $T_xf$ at all $x \in G$, and by the inverse function theorem, this implies that $f\i$ is $C^k$.
Thus, $\Diff_{C^k}(G)^G$ is an open subset of the set $C^k(G,G)^G$ of right-invariant $C^k$ functions, endowed with the compact-open topology of $C(G,J^k(G,G))$.
We next show that the set $C^k(G,G)^G$ of right-invariant $C^k$ functions is a Banach manifold.
Let $\ta:TG\supseteq U\to G$ denote the local addition on $G$, let $\pi:TG\to G$ be the canonical projection.
Recall from \cref{lem:banach} that the space $\mathfrak X_{C^k}(G)^G$ of right-invariant $C^k$ vector fields is Banach, and define the open subset
\[
V = \{X \in \mathfrak X_{C^k}(G)^G: X(e) \in U\}.
\]
For any $f \in C^k(G,G)^G$ and $X \in V$, the function $\ta \o X \o f$ is right-invariant and $C^k$ as a composition of $C^k$ functions.
Thus, the following map is well-defined:
\[
v_f: V\to C^k(G,G)^G,
\qquad
X\mapsto \ta\o X\o f.
\]
It is injective because right-invariant functions are uniquely determined by their value at the identity.
Let $U_f=v_f(V)$.
Then $U_f$ consists of all $h \in C^k(G,G)^G$ such that $(h\o f\i(e),e)\in(\ta,\pi)(U)$.
As $(\ta,\pi)(U)$ is an open subset of $G\times G$, it follows that $U_f$ is open.
The chart around $f$ is defined as $u_f=v_f\i:U_f\to V$, which is a homeomorphism.
For any $g \in C^k(G,G)^G$, the set $u_f(U_f\cap U_g)$ consists of all $X \in V$ such that $(\ta\o X\o f\o g\i(e),e)$ is contained in $(\ta,\pi)(U)$.
Thus, $u_f(U_f\cap U_g)$ is open in $V$.
By symmetry, $u_g(U_f\cap U_g)$ is open, too.
The chart change is given by
\[
u_f\o u_g\i:
\left\{\begin{aligned}
u_g(U_f\cap U_g) &\to u_f(U_g\cap U_g),
\\
X &\mapsto (\ta,\pi)\i \o (\ta\o X\o g\o f\i,\Id_G).
\end{aligned}\right.
\]
The right-hand side is a $C^k$ vector field, which depends smoothly on the $C^k$ vector field $X$ by \cref{lem:Ck_calculus}.
Thus, the chart changes are smooth, and we have shown that $C^k(G,G)^G$ is a smooth Banach manifold.
We define a right-invariant local addition on $C^k(G,G)^G$ by mapping any $X \in T\Diff_{C^k}(G)^G$ with $X(e) \in U$ to $\tau\circ X \in \Diff_{C^k}(G)^G$; one easily verifies that this has the desired properties.
In charts $u_f$ and $u_{g\o f}$, right-composition by $g \in C^k(G,G)^G$ is the identity on $V$,
\begin{align}
u_{f\o g}(\mu^g(u_f(X)))
=
u_{f\o g}(\mu^g(\ta\o X\o f))
=
u_{f\o g}(\ta\o X\o f \o g)
=
X,
\end{align}
and is therefore smooth.
Moreover, the evaluation map is smooth, as can be seen in the chart $u_f$:
\begin{equation}
\on{ev}_e\o u_f\i(X) = \ta\o X\o f(e).
\end{equation}
This concludes the proof.
\end{proof}

\begin{lemma}[Topological group of right-invariant $C^k$ diffeomorphisms]
\label{lem:CkGGG_topological_group}
Let $G$ carry a right-invariant local addition.
Then, for any $k \in \mathbb N$, the space $\Diff_{C^k}(G)^G$ is a topological group.
\end{lemma}

\begin{proof}
Composition $\Diff_{C^k}(G)^G \times \Diff_{C^k}(G)^G \to \Diff_{C^k}(G)^G$ is sequentially continuous by \cref{lem:Ck_calculus}.\ref{lem:Ck_calculus:4}.
Sequential continuity implies continuity because $\Diff_{C^k}(G)^G$ is locally homeomorphic to the metrizable space $\mathfrak X_{C^k}(G)^G$, as shown in \cref{lem:banach}. 

We claim that inversion $\Diff_{C^0}(G)^G\to \Diff_{C^0}(G)^G$ is sequentially continuous (and thus continuous) in the compact-open topology.
Let $f_n\to f$ be a converging sequence in $\Diff_{C^0}(G)^G$.
Then, $f_n(e)\to f(e)$ and $f_n(e)\i\to f(e)\i$ thanks to the continuity of the inversion $G\to G$.
By right-invariance, $f\i(e)=f(e)\i$ and $f_n\i(e)=f_n(e)\i$.
Thus, $f_n\i(e)\to f\i(e)$.
By \cref{lem:multiplication_bounded}, right-translations are uniformly bounded on compacts in $U$, for some sufficiently small chart domain $U$ of $e \in G$.
Thus, $f_n\i\to f\i$ uniformly on compacts in $U$.
Equivalently, $f_n\i \to f\i$ in the compact-open topology on $U$.
By right translations, $f_n\i \to f\i$ in the compact-open topology on $\mu^y(U)$, for any $y \in G$.
As the sets $\mu^y(U)$ form an open cover of $G$, this implies convergence in the compact-open topology on all of $G$.
This proves the claim.

We claim that inversion $\Diff_{C^k}(G)^G\to \Diff_{C^k}(G)^G$ is sequentially continuous (and thus continuous).
Let $f_n\to f$ be a converging sequence in $\Diff_{C^k}(G)^G$, and let $K$ be a compact subset of $G$.
As $f_n\i\to f\i$ in $\Diff_{C^0}(G)^G$, the set $L=f\i(K)\cup\bigcup_n f_n\i(K)$ is compact.
We assume that $K$ and $L$ are contained in chart domains.
This assumption is without loss of generality because any given compact set is a finite union of compact sets such that the desired property holds for sufficiently large $n$.
In the charts, $f_n\i$ converges to $f\i$ uniformly on $K$, and $j^kf_n$ converges to $j^kf$ uniformly on $L$.
Then, $j^k_{f_n\i}f_n$ converges to $j^k_{f\i}f$ uniformly on $K$ because
\begin{equation}
\|j^k_{f_n\i(x)}f_n-j^k_{f\i(x)}f\|
\leq
\|j^k_{f_n\i(x)}f_n-j^k_{f_n\i(x)}f\|+\|j^k_{f_n\i(x)}f-j^k_{f\i(x)}f\|
\end{equation}
where the first summand tends to zero uniformly in $x \in K$ because $j^kf_n$ tends to $j^kf$ uniformly on $L$, and the second summand tends to zero uniformly in $x \in K$ because $j^kf$ is uniformly continuous on $L$.
By \cref{lem:jet_inversion} jet inversion is continuous, and thus uniformly continuous on compacts, and consequently $(j^k_{f_n\i}f_n)\i$ converges to $(j^k_{f\i}f)\i$ uniformly on $K$.
Equivalently, by \cref{def:jet_inversion} of jet inversion, $j^k(f_n\i)$ converges to $j^k(f\i)$ uniformly on $K$.
This proves the claim.
\end{proof}

\begin{lemma}[Banach manifold of right-invariant $k$-jets]
\label{lem:JkGGG}
Let $G$ carry a right-invariant local addition.
Then, the set
\begin{equation}
J^k_e(G,G)^G = \big\{j^k_ef: f \in \Diff_{C^k}(G)^G\big\}
\end{equation}
is a submanifold of $J^k_e(G,G)$, and $\Diff_{C^k}(G)^G$ is diffeomorphic to $J^k_e(G,G)^G$ via the map
\begin{equation}
j^k_e:\Diff_{C^k}(G)^G \ni f \mapsto j^k_e f \in J^k_e(G,G)^G.
\end{equation}
\end{lemma}

\begin{proof}
We will construct a submanifold chart around any given $\sigma=j^k_ef \in J^k_e(G,G)^G$.
Let $U$ be the domain of the right-invariant local addition $\tau:TG\supseteq U \to G$,
let $U^k_f$ be the open set of all $\rho=(e,x,p) \in J^k_e(G,G)$ such that $(x,f(e))\in(\ta,\pi)(U)$,
and let $V^k_f$ be the open set of all $\eta=(f(e),X,q) \in J^k_{f(e)}(TG)$ such that $X \in U$.
Then, the function
\begin{equation}
u^k_f:U^k_f \ni \rho \mapsto j^k(\ta,\pi)\i \bullet (\rho,j^k_ef) \bullet j^k_{f(e)} f\i \in V^k_f
\end{equation}
is smooth with smooth inverse
\begin{equation}
(u^k_f)\i: V^k_f \ni \eta \mapsto j^k\ta \bullet \eta \bullet j^k_ef \in U^k_f.
\end{equation}
Moreover, $u^k_f$ restricts to right-invariant jets as follows:
\begin{equation}
u^k_f:U^k_f\cap J^k_e(G,G)^G \to V^k_f\cap J^k_{f(e)}(TG)^G.
\end{equation}
By \cref{lem:banach}, $\mathfrak X_{C^k}(G)^G$ is a closed linear subspace of $\mathfrak X_{C^k}(G)$, and as these spaces are isomorphic to jet spaces, $J^k_{f(e)}(TG)^G$ is a closed linear subspace of $J^k_{f(e)}(TG)$.
Thus, $u^k_f$ is a submanifold chart around $j^k_ef\in J^k_e(G,G)^G$, and we have shown that $J^k_e(G,G)^G$ is a submanifold of $J^k_e(G,G)$.

It remains to show that this submanifold is diffeomorphic to $\Diff_{C^k}(G)^G$.
In the chart for $\Diff_{C^k}(G)^G$ constructed in \cref{lem:CkGGG_manifold},
\begin{align}
u_f:C^k(G,G)^G \supseteq U_f&\to V_f \subseteq \mathfrak X_{C^k}(G)^G,
\end{align}
the map $j^k_e$ is the identification of right-invariant vector fields with their $k$-jets at $e \in G$:
\begin{equation}
u^k_f\o j^k_e\o u_f^{-1}: V_f \ni X \mapsto j^k_{f(e)}X \in V^k_f \cap J^k_{f(e)}(TG)^G.
\end{equation}
This is a continuous linear map with continuous inverse by \cref{lem:banach}.
\end{proof}

\begin{lemma}[Push-forwards]
\label{lem:Ckl}
Let $G$ carry a right-invariant local addition.
For any $k,\ell \in \mathbb N$ and $f \in C^{k+\ell}(G,G)^G$, the push-forward along $f$ is a $C^\ell$ function
\begin{equation}
\Diff_{C^k}(G)^G  \ni g \mapsto f\o g \in \Diff_{C^k}(G)^G.
\end{equation}
\end{lemma}

\begin{proof}
By \cref{lem:JkGGG}, $\Diff_{C^k}(G)^G$ is diffeomorphic to $J^k_e(G,G)$ via the identification of right-invariant functions with their $k$-jets at $e \in G$.
Thus, one has to show that the function
\begin{equation}
J^k_e(G,f):J^k_e(G,G)^G\ni \si \mapsto j^k_{\be(\si)}f\bullet\si \in J^k_e(G,G)^G
\end{equation}
is $C^\ell$.
Here, $J^k_e(G,\cdot)$ is viewed as a covariant functor, and $J^k_e(G,f)$ is the functor applied to the morphism $f$.

We claim that $J^k_e(G,f)$ is Gateaux-$C^\ell$, i.e., that the map
\begin{equation}
\label{equ:TlJkGf}
T^\ell J^k_e(G,f):T^\ell J^k_e(G,G)^G\to T^\ell J^k_e(G,G)^G
\end{equation}
is continuous.
The functors $T^\ell$ and $J^k_e(G,\cdot)$ commute up to a natural isomorphism $\kappa$, which is called canonical flip.
Abstractly, at least in finite dimensions, this is a consequence of these functors being product-preserving \cite[Sections~36--37]{kolar2013natural}, but it can also be verified easily by hand.
Thus, the claim is equivalent to the continuity of
\begin{equation}
\label{equ:JkGTlf}
J^k_e(G,T^\ell f)=\kappa\i\o T^\ell J^k_e(G,f)\o\kappa: J^k_e(G,T^\ell G)^G\ni \sigma\mapsto j^k_{\be(\si)}T^\ell f\bullet \sigma \in J^k_e(G,T^\ell G)^G.
\end{equation}
This follows from the continuity of jet composition and the continuity of $j^kT^\ell f$.
Thus, we have shown that $J^k_e(G,f)$ is Gateaux-$C^\ell$, as claimed.

The Fr\'echet-$C^\ell$ property of $J^k_e(G,f)$ is equivalent to uniform continuity of the function  \eqref{equ:TlJkGf} on bounded subsets with compact base, i.e., on subsets $B$ of the bundle $\pi:T^\ell J^k_e(G,G)^G\to J^k(G,G)^G$ such that $\pi(B)$ is compact and such that the vertical component of $B$ is uniformly bounded in any bundle chart.
Similarly to before, by commutation of the functors $T^\ell$ and $J^k_e(G,\cdot)$, this is equivalent to the uniform continuity of the function \eqref{equ:JkGTlf} on bounded subsets with compact base.
This is a consequence of the following facts:
jet composition is uniformly continuous on bounded subsets with compact base by \cref{lem:jet_composition},
and the map $j^kT^\ell f$ is uniformly continuous on bounded subsets with compact base because $f$ is $C^{k+\ell}$ in the sense of Fr\'echet.
\end{proof}

This concludes our investigation of right-invariant $C^k$ diffeomorphisms on $G$.
We next transfer these results to the group of $C^k$ elements in $G$ by identifying any $C^k$ element $x \in G^k$ with the left translation $\mu_x \in \Diff_{C^k}(G)^G$.

\begin{proof}[Proof of \cref{thm:Gk}]
Left-multiplication is a bijection
\[
G^k \ni x \mapsto \mu_x \in \Diff_{C^k}(G)^G,
\]
whose inverse is the evaluation map
\[
\Diff_{C^k}(G)^G \ni f \mapsto f(e) \in G^k.
\]
Via this bijection, multiplication in $G^k$ corresponds to composition in $\Diff_{C^k}(G)^G$, inversion in $G^k$ corresponds to inversion in $\Diff_{C^k}(G)^G$, and the inclusion $G^k\to G$ corresponds to the evaluation map $\Diff_{C^k}(G)^G\to G$ at $e \in G$.
Thus, by declaring the bijection to be a diffeomorphism, we obtain from \cref{lem:CkGGG_manifold,lem:CkGGG_topological_group,lem:Ckl} that $G^k$ is a Banach half-Lie group, which is smoothly included in $G$, and that $G^{k+\ell}$ is contained in $(G^k)^\ell$.
Moreover, the tangent space $T_eG^k$ is identified with the tangent space $T_e\Diff_{C^k}(G)^G$, i.e., with the space of right-invariant $C^k$ vector fields on $G$.
Under this identification, the right-invariant local addition on $G^k$ is obtained by restricting the right-invariant local addition on $G$. 
In the locally defined charts $\tau^{-1}:G\to T_eG$ and $(\tau,\tau)^{-1}\o(\pi,\tau):TG\to T_eG\times T_eG$, the map $\tau:TG\to G$ is continuous linear and restricts to a locally defined linear map $\tau:TG^k\to G^k$. 
By the closed graph theorem, the restriction $\tau:TG^k\to G^k$ is continuous linear in the local charts, hence smooth.
\end{proof}

\begin{proof}[Proof of \cref{lem:smooth_elements}]
$\Diff_{C^\infty}(G)^G$ is an ILB manifold because the charts and chart changes for $\Diff_{C^0}(G)^G$ constructed in \cref{lem:CkGGG_manifold} restrict to charts and chart changes for $\Diff_{C^k}(G)^G$, $k \in \mathbb N$. Multiplication $G^\infty\x G^\infty\to G^\infty$ is smooth because composition $\Diff_{C^\infty}(G)^G\x \Diff_{C^\infty}(G)^G\to \Diff_{C^\infty}(G)^G$ is smooth by convenient calculus \cite[Corollary~3.13]{KrieglMichor97}.
\end{proof}

\newcommand\Zbl[1]{Zbl~\href{https://zbmath.org/#1}{#1}}

\end{document}